\documentclass{nmj}
\usepackage{amscd, amssymb, amsmath, mathrsfs}
\usepackage{tikz-cd}
\usepackage{hyperref}

\newcommand{\Hom}{{\rm{Hom}}}
\newcommand{\Ext}{{\rm{Ext}}}

\newcommand{\Coker}{{\rm{coker}}}
\newcommand{\Fac}{{\rm{Fac}}}
\newcommand{\Filt}{{\rm{Filt}}}
\newcommand{\Sub}{{\rm{Sub}}}
\newcommand{\add}{{\rm{add}}}
\newcommand{\mmod}{{\rm{{mod\mbox{-}}}}}

\newcommand{\CA}{\mathcal{A}}
\newcommand{\C}{\mathcal{C}}

\newcommand{\F}{\mathcal{F}}

\newcommand{\CP}{\mathcal{P}}
\newcommand{\Q}{\mathcal{Q}}
\newcommand{\T}{\mathcal{T}}
\newcommand{\CU}{\mathcal{U}}
\newcommand{\X}{\mathcal{X}}

\def\k{\Bbbk}

\newcommand{\M}{\mathscr{M}}
\newcommand{\U}{\mathscr{U}}

\newcommand{\Rep}[1]{%
  {%
    \small%
    \begin{matrix}%
      #1%
    \end{matrix}%
  }%
}

\newcommand{\rep}[1]{%
  {%
    \tiny%
    \begin{matrix}%
      #1%
    \end{matrix}%
  }%
}

\newcommand{\ind}{{\rm{ind\mbox{-}}}}
\newcommand{\Mod}{{\rm{Mod\mbox{-}}}}
\newcommand{\Pc}{{P_{\bullet}}}

\newcommand{\rt}{\rightarrow}
\newcommand{\lrt}{\longrightarrow}
\newcommand{\st}{\stackrel}

\theoremstyle{plain}
\newtheorem{theorem}{Theorem}[section]
\newtheorem{proposition}[theorem]{Proposition}
\newtheorem{lemma}[theorem]{Lemma}
\newtheorem{corollary}[theorem]{Corollary}

\theoremstyle{remark}
\newtheorem{remark}[theorem]{Remark}
\newtheorem{example}[theorem]{Example}

\theoremstyle{definition}
\newtheorem{definition}[theorem]{Definition}

\newcommand\blfootnote[1]{%
  \begingroup
  \renewcommand\thefootnote{}\footnote{#1}%
  \addtocounter{footnote}{-1}%
  \endgroup
}

\begin{document}

\title[On Higher Torsion Classes]{On Higher Torsion Classes}

\author{
  \name[J. Asadollahi]{J. Asadollahi}
  \address{Department of Pure Mathematics, Faculty of Mathematics and Statistics, University of Isfahan, P.O.Box: 81746-73441, Isfahan, Iran}
  \email{asadollahi@sci.ui.ac.ir, asadollahi@ipm.ir}
}

\author{
  \name[P. J\o rgensen]{P. J\o rgensen}
  \address{Department of Mathematics, Aarhus University, Ny Munkegade 118, 8000 Aarhus C, Denmark}
  \email{peter.jorgensen@math.au.dk}
}

\author{
  \name[S. Schroll]{S. Schroll}
  \address{Department of Mathematics, University of Cologne, Weyertal 86-90, 50931 Cologne, Germany}
  \email{schroll@math.uni-koeln.de}
}

\author{
  \name[H. Treffinger]{H. Treffinger}
  \address{Universit\'{e} Paris Cit\'e, B\^{a}timent Sophie Germain 5, rue Thomas Mann 75205, Paris Cedex 13, FRANCE}
  \email{treffinger@imj-prg.fr}
}


\maketitle

\begin{abstract}
Building on the embedding of an $n$-abelian category $\M$  into an abelian category $\CA$ as an $n$-cluster-tilting subcategory of $\CA$, in this paper we relate the $n$-torsion classes of $\M$ with  the torsion classes of $\CA$.
Indeed, we show that every $n$-torsion class in $\M$ is given by the intersection of a torsion class in $\CA$ with $\M$.
 Moreover, we show that every chain of $n$-torsion classes in the $n$-abelian category $\M$ induces a Harder-Narasimhan filtration for every object of $\M$. We use the relation between $\M$ and $\CA$ to show that every Harder-Narasimhan filtration induced by a chain of $n$-torsion classes in $\M$ can be induced by a chain of torsion classes in $\CA$.
 Furthermore, we show that $n$-torsion classes are preserved by Galois covering functors, thus we provide a way  to systematically  construct new (chains of) $n$-torsion classes.
\end{abstract}

\section{Introduction}
Higher homological algebra has its origin in the study of  $n$-cluster-tilting subcategories of abelian and triangulated categories  in \cite{Iyama2007a, Iyama2007}.
\blfootnote{JA was supported by a LMS Grant - No. 51805 to visit the University of Leicester, where part of this work is done and by the Iran National Science Foundation (INSF) under project No. 4001480.
PJ was supported by a DNRF Chair from the Danish National Research Foundation (grant number DNRF156), by Aarhus University Research Foundation (grant no. AUFF-F-2020-7-16), a Research Project 2 from the Independent Research Fund Denmark (grant no. 1026-00050B), and by the Engineering and Physical Sciences Research Council (EPSRC grant number EP/P016014/1).
SS is supported and HT is partially supported by the EPSRC through the Early Career Fellowship, EP/P016294/1.
HT is supported by the European Union’s Horizon 2020 research and innovation programme under the Marie Sklodowska-Curie grant agreement No 893654. HT is partially supported by the Deutsche Forschungsgemeinschaft under Germany's Excellence Programme Grant - EXC-2047/1-390685813.}
The subject  has greatly developed  since its introduction with more and more of the classical notions emerging in the higher setting.
The key idea of higher homological algebra is  the study of categories where the shortest non-split exact sequences are composed of $n+2$ objects, for a fixed positive integer $n$.
In particular, $1$-homological algebra corresponds to the classical theory of abelian, exact and triangulated categories and their classical generalisations such as quasi-abelian and extriangulated categories.

In recent years the  importance of higher homological algebra is starting to emerge through articles showing connections between this subject and other branches of the mathematical sciences, such as combinatorics and homological mirror symmetry \cite{Dickerhoff2019, Herschend2011, Iyama2011, Oppermann2012, WilliamsarXiv:2007.12664}.

Since its inception, it has been shown that many of the fundamental homological concepts in the classical theory have an analogue in higher homological algebra.
Classical homological algebra is a by now well-developed subject and many of the fundamental concepts have several equivalent definitions characterising different properties and aspects of the various concepts.
However, this poses a difficulty in generalising these ideas into the setting of higher homological algebra, since the classically equivalent definitions might lift to non-equivalent concepts in the higher setting.
Therefore,  even if $n$-exact sequences are easy to identify, the search for the best definition for higher analogues of classical notions is not an easy task.
A breakthrough in this direction was achieved in \cite{Jasso2016}, where the definitions of $n$-abelian and $n$-exact categories were introduced and in \cite{Ebrahimi2020, Kvamme2020} where it is shown that any $n$-abelian category arises as an $n$-cluster tilting subcategory of an abelian category.

A key notion in representation theory and homological algebra is the concept of torsion theories, introduced by Dickson in \cite{Dickson1966}.
Their natural relevance, for example, in relation to derived categories and tilting theory, have led to the use of homological algebra in many branches of mathematics, including algebraic geometry and mathematical physics.
Torsion theory is built on the notion of torsion pairs, where a torsion pair is a pair of full subcategories $(\T, \F)$ with no non-zero morphism from the torsion class $\T$ to the torsionfree class $\F$.
A definition of torsion classes in higher homological algebra has recently been given by the second author in \cite{Jorgensen2016} based on the classical characterisation of the existence of a unique torsion subobject and a unique torsion free quotient for every object in the category.

Part of our motivation for writing this paper is to show that when considering an $n$-abelian category in the context of its ambient abelian category, that is viewed as an $n$-cluster tilting subcategory of an abelian category, then the definition of higher torsion classes in \cite{Jorgensen2016} seems to encode all the relevant properties one would expect.
More precisely, one of the main ideas of the paper is built on the comparison of $n$-torsion classes in an $n$-abelian category and the corresponding minimal torsion classes generated by these $n$-torsion classes in an abelian category which contains the $n$-abelian category as a $n$-cluster-tilting subcategory. In particular, we do this in the context of Harder-Narasimhan filtrations and $n$-Harder-Narasimhan filtrations  which we define   and  which we also study  in the context of Galois coverings.

Our first result characterises the minimal torsion classes in an abelian category containing a $n$-torsion class.
Note that for every $n\geq 1$, the set of all $n$-torsion classes in an $n$-abelian category $\M$ forms a poset under the natural order given by the inclusion.
See Corollary \ref{cor:poset} and Theorem \ref{thm:torsfromntors}.

\begin{theorem}\label{thm:1.1}
Let $\M$ be an $n$-cluster-tilting subcategory of a skeletally small abelian length category $\CA$.
Then there is an injective morphism of posets
\[ T: \{ \text{$n$-torsion classes in $\M$} \} \rightarrow \{ \text{torsion classes in $\CA$} \}  \]
given by sending an $n$-torsion class $\U$ in $\M$  to the minimal torsion class  in $\CA$ containing $\U$.
Moreover, a torsion class $\T$ in $\CA$ is of the form $T(\U)$, for $\U$ an $n$-torsion class in $\M$, if and only if the following hold:
\begin{enumerate}
    \item $tM \in \U$ for all $M \in \M$; where $t=t_{\T}$ is the torsion functor associated to the torsion class $\T$;
    \item $\T$ is the minimal torsion class in $\CA$ containing $\{ tM : M \in \M \}$;
    \item $\mbox{\rm Ext}^{n-1}_\CA (X, Y) = 0$, for all $X \in \{ tM : M \in \M \}$ and $Y \in \{{\rm coker}(tM \hookrightarrow M) \mid M \in \M\}$.
\end{enumerate}
In this case, $\U= \T \cap \M=\{ tM : M \in \M \}$.
\end{theorem}

Stability conditions were introduced in \cite{Mumford1965} to attack algebro-geometric problems.
Given their simplicity and effectiveness, their definition was later adapted to other contexts, such as quiver representations \cite{King1994, Schofield1991}, abelian categories \cite{Rudakov1997} and triangulated categories \cite{Bridgeland2007}.

One of the features of stability conditions relies on the fact that every stability condition determines for each non-zero object in a category a stratification by more well-behaved objects.
This stratification, usually known as the \textit{Harder-Narasimhan filtration}, has been used to make possible calculations that otherwise would be highly complicated or even impossible. Applications of this can be found, for example, in the study of Donaldson-Thomas invariants and in the mirror symmetry program \cite{Bridgeland2018, Reineke2003}.

In a recent paper  \cite{T-HN-filt}, the fourth author has introduced an axiomatic approach to Harder-Narasimhan filtrations for abelian categories, by showing that the existence of such filtration for every object in an abelian category is equivalent to the existence of a chain of torsion classes in the category.
Since this construction of Harder-Narasimhan filtrations does not depend on the existence of a stability condition, it allows the introduction of Harder-Narasimhan filtrations to non-abelian settings such as quasi-abelian categories \cite{Tattar2019}.
In the second main result of this paper we push this idea further by showing that chains of $n$-torsion classes  induce Harder-Narasimhan filtrations in $n$-abelian categories.
Moreover, we  show that the Harder-Narasimhan filtrations obtained in this way coincide with the Harder-Narasimhan filtrations in the ambient abelian category which contains the $n$-abelian category as an $n$-cluster tilting subcategory.
See Theorem \ref{thm:nHNfilt} and Theorem \ref{thm:comparison}.

\begin{theorem}
\label{thm:1.2}
Let $\M$ be an $n$-abelian category and let $\delta$ be a chain of $n$-torsion classes in $\M$.
Then for $M$ a non-zero object in $\M$ the following hold:
\begin{enumerate}
\item $\delta$ induces an $n$-Harder-Narasimhan filtration of $M$ which is unique up to isomorphism.
\item If $\M$ is the $n$-cluster-tilting subcategory of a skeletally small abelian length category $\CA$  then the $n$-Harder-Narasimhan filtration of $M$ in $\M$ induced by $\delta$ is equal to the Harder-Narasimhan filtration of $M$ in $\CA$ induced by $T(\delta)$, where $T(\delta)=\{T(\U) \mid \U \in \delta\}$.
\end{enumerate}
\end{theorem}

Note that we need the `skeletally small' assumption in Theorems \ref{thm:1.1} and \ref{thm:1.2}(2) because for these results we rely on Theorem 7.3 of \cite{Kvamme2020}, also Theorem 4.3 of  \cite{Ebrahimi2020}, showing that under this assumption, any $n$-abelian category is embedded an abelian category as an $n$-cluster-tilting subcategory of $\CA$, see Theorem \ref{thm:embedding} below.

\smallskip

Another important concept in representation theory is the notion of Galois coverings, introduced by Gabriel in \cite{Gabriel, GabrielRoiter} and studied further by many authors since.
The initial aim was to reduce a problem for modules over an algebra $A$ to that of a category $\C$ with an action of a group $G$ such that $A$ is equivalent to the orbit category $\C/G$.
The theory has much evolved since its inception leading to a vast body of literature on the subject \cite{BG, Rid, Green, Asashiba, M-VdlP, BautistaLiu, DarpoIyama}.
In particular, it has been shown that several nice properties, such as local finiteness or Cohen-Maculay finiteness, are preserved by Galois coverings \cite{Gabriel, AHV}.
Recently, Darp\"{o} and Iyama show in \cite{DarpoIyama} that $n$-cluster-tilting subcategories are, under certain conditions, preserved by Galois coverings.
Their construction is based on the fact  that under certain technical conditions, which are described in Theorem~1.3, given a Galois covering functor $P: \C \lrt  \C/G$, there exists a Galois precovering functor $P_{\bullet}: \mmod \C \lrt \mmod(\C/G)$, called push-down functor, between the categories of finitely presented functors over $\C$ and $\C/G$ such that $P_{\bullet}(\M)$ is an $n$-cluster-tilting subcategory in $\mmod\C/G$, where $\M$ is a certain $n$-cluster tilting subcategory of $\mmod \C$.
We add to this by showing that under similar conditions as in \cite{DarpoIyama}, $n$-torsion classes, chains of $n$-torsion classes, and $n$-Harder-Narasimhan filtrations are preserved by Galois coverings.
Note that some authors use `G-covering' instead of `Galois covering' for this generalized version of the classical Galois covering theory, see e.g. \cite{Asashiba}.
More precisely, we show the following. Recall that the action of a group $G$ on $\C$ is called admissible if $gX\ncong X$ for each indecomposable object $X$ in $\C$ and each $g\neq 1$ in $G$, see Definition \ref{def:admissible} below.
See Theorem~\ref{thm:cov-main} and Proposition~\ref{prop:pushdownHN}.

\begin{theorem}
Let  $\C$ be a small locally bounded Krull-Schmidt $\k$-category  with an admissible action of a group $G$ on $\C$ inducing an admissible action on $\mmod \C$.
Suppose that $\M$ is an $n$-cluster tilting $G$-equivariant full subcategory of $\mmod \C$ such that $P_{\bullet}(\M)$ is functorially finite in $\mmod(\C/G)$.
If $\U$ is a $G$-equivariant $n$-torsion class of $\M$ then $P_{\bullet}(\U)$ is an $n$-torsion class of   $P_{\bullet}(\M)$.

Moreover, if $\U$ is a $G$-equivariant $n$-torsion class in $\M$ then the following statements hold for $M \in \M$.
\begin{enumerate}
	\item  An object $U^M$  in $\U$ is the torsion object of $M$ with respect to $\U$ if and only if $\Pc(U^M)$ is the torsion object of $\Pc(M)$ with respect to $\Pc(\U)$.
	\item If $\delta=\{\U_s : s\in [0,1]\}$ is a chain of $G$-equivariant $n$-torsion classes in $\M$ then
		$$0 =M_0 \subsetneq M_1 \subsetneq \dots \subsetneq M_{r-1} \subsetneq M_r = M$$
		is the $n$-Harder-Narasimhan filtration of $M$ with respect to $\delta$ in $\M$ if and only if
		$$0 =\Pc(M_0) \subsetneq \Pc(M_1) \subsetneq \dots \subsetneq \Pc(M_{r-1}) \subsetneq \Pc(M_r )= \Pc(M)$$
		is the $n$-Harder-Narasimhan filtration of $\Pc(M)$ with respect to the chain of $n$-torsion classes $\Pc(\delta)$ in $\Pc(\M)$.

	\item If $T(\U)$ is $G$-equivariant then $T(\Pc(\U))=\Pc(T(\U))$, that is the following diagram is commutative

{\small
\begin{tikzcd}
  \text{  \{$G$-equivariant $n$-torsion classes in $\M$ \}} \ar{r}{T(-)} \ar{dd}{\Pc(-)} & \text{\{$G$-equivariant torsion classes in $\mmod\C$\}}  \ar{dd}{\Pc(-)}\\
 	&  & \\
\text{ \{ $n$-torsion classes in $\Pc(\M)$ \}} \ar{r}{T(-)} & \text{\{torsion classes in $\mmod\C/G$ \}}
\end{tikzcd}}

\end{enumerate}
\end{theorem}

{\bf Acknowledgements: } 
JA would like to thank the third and fourth authors for their warm hospitality and also excellent mathematical discussions during his visit. 
The authors thank the referee for the careful reading of the paper and for insightful comments.

\section{Background}\label{sec:background}

An abelian category $\CA$ is said to be a \textit{length category} if every object of $\CA$ is of finite length, that is, every object has a finite filtration starting with the zero object  such that at each step the quotient is a simple object. Such a filtration usually is called a composition series or a Jordan-H\"{o}lder sequence of the object.
A category $\CA$ is called \textit{skeletally small} if the class of all isomorphism classes of objects in $\CA$ is a set.
In this paper, whenever we say that $\CA$ is an abelian category, we assume that $\CA$ is a skeletally small abelian length category.

Given a full subcategory $\X$ of $\CA$ which is closed under direct sums, we define the subcategory $\Fac(\X)$ of $\CA$ to be the full subcategory of $\CA$ whose objects are quotients of objects in $\X$,
$$\Fac (\X) = \{Y \in \CA : \exists \ \text{exact sequence} \ X \rt Y \rt 0, \text{ for some $X \in \X$}\}.$$
Similarly, the category $\Sub(\X)$ is the full subcategory of $\CA$ whose objects are subobjects of objects in $\X$,
$$\Sub (\X) = \{Y \in \CA :\exists \ \text{exact sequence} \  0 \rt Y \rt X, \text{ for some $X \in \X$}\}.$$
We say that $\X$ is a \textit{generating} subcategory of $\CA$ if $\Fac(\X)=\CA$.
Dually, we say that $\X$ is \textit{cogenerating} if $\Sub(\X)=\CA$.

We say that an object $M$ of $\CA$ is filtered by $\X$ if there exists a finite sequence of subobjects
$$M_0 \subset M_1 \subset \dots \subset M_n$$
such that $M_0=0$, $M_n =M$ and $M_i/M_{i-1} \in \X$ for all $1 \leq i \leq n$.
We denote by $\Filt{\X}$ the full subcategory of all objects filtered by $\X$.
Note that $\X$ is a full subcategory of $\Fac(\X)$, $\Sub(\X)$ and $\Filt(\X)$.

\subsection{$n$-cluster-tilting subcategories and $n$-abelian categories}
Let $n$ be an integer greater than or equal to $1$.
The theory of higher homological algebra started in \cite{Iyama2007a,Iyama2007} with the study of the so-called \textit{$n$-cluster-tilting subcategories} of module categories. Their definition for arbitrary abelian categories is the following.

Let us preface the definition by recalling some notions.
Let $\X$ be a full subcategory of $\CA$.
We say that $\X$ is a contravariantly finite subcategory of $\CA$ if every object $A \in \CA$, admits a right $\X$-approximation, that is, for every $A \in \CA$ there exists a morphism $\pi : M \rt A$ with $M \in \X$ such that any other morphism $\pi': M' \rt A$, with $M' \in \X$, factors through $\pi$.
Dually, the notion covariantly finite subcategories is defined.
A functorially finite subcategory of $\CA$, is a subcategory which is both contravariantly and covariantly finite.

\begin{definition}\cite[Definition 3.14]{Jasso2016}
Let $\CA$ be an abelian category.
A functorially finite generating-cogenerating subcategory $\M$ of $\CA$ is \textit{$n$-cluster-tilting} if
$$\M  = \{X \in \CA : \Ext_{\CA}^i(X,M)=0 \text{ for all $M \in \M$ and all $1\leq i \leq n-1$}\}$$
$$ \ \ \ \ \  = \{Y \in \CA : \Ext_{\CA}^i(M,Y)=0 \text{ for all $M \in \M$ and all $1\leq i \leq n-1$}\}.$$
\end{definition}

The concept of $n$-abelian category was introduced in \cite{Jasso2016} as a generalisation of the classical concept of abelian categories, to formalize the homological structure of $n$-cluster-tilting subcategories.
The formal definition uses the notions of $n$-kernel and $n$-cokernel of a morphism, that we now recall.
Let $f^0: X^0 \rt X^1$ be a morphism in an additive category $\M$. A sequence of morphisms
$$X^1 \xrightarrow{f^1} X^2 \xrightarrow{f^2} \dots \xrightarrow{f^{n-1}} X^{n} \xrightarrow{f^n} X^{n+1}$$
 is called an \emph{$n$-cokernel of $f^0$} if, for every $M \in \M$, the following sequence
$$0 \rt \M(X^{n+1}, M) \rt \M(X^n, M) \rt \cdots \rt \M(X^1, M) \rt \M(X^0, M) $$
of abelian groups is exact. An $n$-cokernel of $f^0$ is denoted by $(f^1, \dots, f^{n})$.  The notion of \emph{$n$-kernel of a morphism} is defined dually. The sequence
$$X^0 \xrightarrow{f^0} X^1 \xrightarrow{f^1} \dots \xrightarrow{f^{n-1}} X^{n} \xrightarrow{f^n} X^{n+1}.$$
is called \emph{$n$-exac}t if $(f^1, \dots, f^{n})$  is an $n$-cokernel of $f^0$ and $(f^0, \dots, f^{n-1})$ is an $n$-kernel of $f^n$.

\begin{definition}\cite[Definition\;3.1]{Jasso2016}
Let $n$ be a positive integer.
An additive category $\M$ is \textit{$n$-abelian} if the following axioms hold:
\begin{enumerate}
\item[(A0)] $\M$ has split idempotents;
\item[(A1)] Every morphism in $\M$ has an $n$-kernel and an $n$-cokernel;
\item[(A2)] For every monomorphism $f^0: X^0 \rt X^1$ and any $n$-cokernel $(f^1, \dots, f^{n})$ {of $f^0$}, the following sequence is $n$-exact
$$X^0 \xrightarrow{f^0} X^1 \xrightarrow{f^1} \dots \xrightarrow{f^{n-1}} X^{n} \xrightarrow{f^n} X^{n+1};$$
\item[(A2$^{op}$)] For every epimorphism $f^n: X^{n} \rt X^{n+1}$ and any $n$-kernel $(f^0, \dots, f^{n-1})$ of $f^n$, the following sequence is $n$-exact
$$X^0 \xrightarrow{f^0} X^1 \xrightarrow{f^1} \dots \xrightarrow{f^{n-1}} X^{n} \xrightarrow{f^n} X^{n+1}.$$
\end{enumerate}
\end{definition}

The motivating example for $n$-abelian categories are $n$-cluster-tilting subcategories and indeed as stated below it is now known that these are the only small $n$-abelian categories.

\begin{theorem}\cite[Theorem 3.16]{Jasso2016}
Let $\M$ be an $n$-cluster-tilting subcategory of an abelian category $\CA$.
Then $\M$ is $n$-abelian.
\end{theorem}

It is worth noticing that all $n$-exact sequences in $\M$ are the $n$-extensions of $\CA$ where all terms of the extensions are in $\M$.
The converse of the previous result also holds.

\begin{theorem} \cite[Theorem 4.3]{Ebrahimi2020} \cite[Theorem 7.3]{Kvamme2020} \label{thm:embedding}
Let $\M$ be an $n$-abelian category.
 Then there exists an abelian category $\CA$ and a fully faithful functor $F: \M \rt \CA$ such that $\F(\M)$ is an $n$-cluster-tilting subcategory of $\CA$.
\end{theorem}

Before going any further we introduce a running example that will help us to illustrate most of the results of this paper.

\begin{example}\label{ex:running}
Let $A$ be the path algebra of the quiver
\begin{tikzcd}
1 \ar{r}{\alpha} & 2 \ar{r}{\beta} & 3
\end{tikzcd}
modulo the ideal generated by the relation $\alpha\beta$.
The Auslander-Reiten quiver of $A$ can be seen in Figure~\ref{fig:Ar-quiverA}, where the dotted arrows correspond to the Auslander-Reiten translations in $\mmod A$.
It is well-known that the subcategory
$$\M = \add\left\{ \rep{3} \oplus \rep{2\\3} \oplus \rep{1\\2} \oplus \rep{1}	 \right\} $$
is a $2$-cluster tilting subcategory of $\mmod A$.
In Figure~\ref{fig:Ar-quiverA} you can find in red the indecomposable objects of $\mmod A$ that belong to $\M$.

\begin{figure}
    \centering
			\begin{tikzpicture}[line cap=round,line join=round ,x=2.0cm,y=1.8cm, scale = 1.5]
				\clip(-2.2,0.9) rectangle (2.1,2.5);
					\draw [->] (-1.8,1.2) -- (-1.2,1.8);
					\draw [->] (0.2,1.2) -- (0.8,1.8);
					\draw [<-, dashed] (-1.8,1.0) -- (-0.2,1.0);
					\draw [<-, dashed] (0.2,1.0) -- (1.8,1.0);
					\draw [->] (-0.8,1.8) -- (-0.2,1.2);
					\draw [->] (1.2,1.8) -- (1.8,1.2);
				
				\begin{scriptsize}
					\draw[color=red] (-2,1) node {$\Rep{3}$};
					\draw[color=black] (0,1) node {$\Rep{2}$};
					\draw[color=red] (2,1) node {$\Rep{1}$};
					\draw[color=red] (-1,2) node {$\Rep{2\\3}$};
					\draw[color=red] (1,2) node {$\Rep{1\\2}$};
				\end{scriptsize}
			\end{tikzpicture}
\caption{The Auslander-Reiten quiver of $A$}
    \label{fig:Ar-quiverA}
\end{figure}
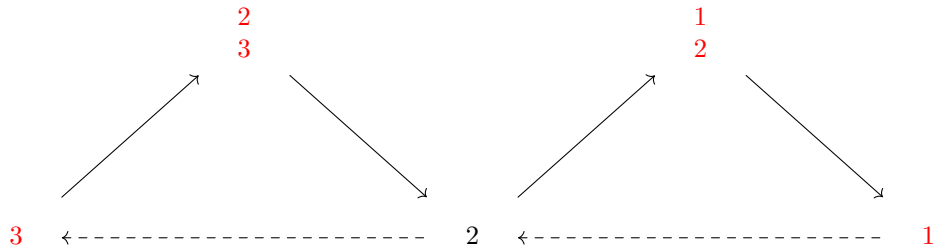

\end{example}

\subsection{Torsion and $n$-torsion classes}

Generalising the classical properties of abelian groups, Dickson introduced in \cite{Dickson1966} the notion of \textit{torsion pair} as follows.

\begin{definition}
Let $\CA$ be an abelian category.
Then the pair $(\T, \F)$ of full subcategories of $\CA$ is a \textit{torsion pair} if the following conditions are satisfied:
\begin{itemize}
\item $\Hom_{\CA} (X,Y)=0$ for all $X \in \T$ and $Y \in \F$.
\item For every module $M$ in $\CA$ there exists a short exact sequence
\[0 \rightarrow tM \xrightarrow{ \iota_M } M \xrightarrow{ \pi_M } fM \rightarrow 0\]
where $tM \in\mathcal{T}$ and $fM \in\mathcal{F}$.
\end{itemize}
This short exact sequence is unique up to isomorphisms and is known as the \textit{canonical short exact sequence} of $M$ with respect to $(\T, \F)$.
Moreover we say that $\T$ is a \textit{torsion class} and $\F$ is a \textit{torsion free class}.
\end{definition}

In the same paper \cite{Dickson1966} where he introduced the concept of torsion pair, Dickson gave an useful characterisation of torsion and torsion free classes.

\begin{theorem}\cite[Theorem 2.3]{Dickson1966}
A full subcategory $\T$ of an abelian category $\CA$ is a torsion class if and only if $\T$ is closed under factors and extensions.
Dually, a full subcategory $\F$ of an abelian category $\CA$ is a torsion free class if and only if $\F$ is closed under subobjects and extensions.
\end{theorem}

We denote by $\text{tors}(\CA)$ the set of all torsion classes in $\CA$. It is clear that the natural inclusion of sets induces a natural partial order in $\text{tors}(\CA)$.

Given a subcategory $\X$ of $\CA$, we denote by $T(\X)$ the minimal torsion class of $\CA$ containing $\X$.
It is well-known that $T(\X)$ coincides with all the objects of $\CA$ filtered by elements in the category $\Fac (\X)$, that is,  $T(\X)= \Filt(\Fac(\X))$. For a proof see \cite[Proposition 2.1]{Thomas2021}.

With the development of higher homological algebra, it is natural to consider higher analogues of torsion classes in this framework. The first such notion is introduced in \cite{Jorgensen2016}. The formal definition is as follows.

\begin{definition}\cite[Definition 1.1]{Jorgensen2016}\label{def:ntorsionclass}
Let $\M$ be an $n$-abelian category.
A full subcategory $\U$ of $\M$ is an \textit{$n$-torsion class} if for every $M\in \M$ there exists an $n$-exact sequence
\begin{equation}
  0 \xrightarrow{} U^M \xrightarrow{} M \xrightarrow{} V^1 \xrightarrow{ v^1 } \cdots \xrightarrow{ v^{ n-1 } } V^n \xrightarrow{} 0,
\end{equation}
where $U^M$ is an object of $\U$ and the sequence
\begin{equation}
  0 \xrightarrow{} \Hom_\M(U,V^1) \xrightarrow{ } \Hom_\M(U,V^2) \xrightarrow{  } \cdots \xrightarrow{  } \Hom_\M(U,V^n) \xrightarrow{} 0
\end{equation}
is exact, for all objects $U$ in $\U$. $U^M$ is called the $n$-torsion subobject of $M$ with respect to $\U$.
\end{definition}

\subsection{Harder-Narasimhan filtrations in abelian categories}

Inspired by the relation between stability conditions and torsion classes,  in \cite{T-HN-filt} the relation between Harder-Narasimhan filtrations and torsion classes was studied.
This was done through the introduction of chains of torsion classes as follows.

\begin{definition}\cite[Definition 2.1]{T-HN-filt}
A chain of torsion classes $\eta$ in an abelian category $\CA$ is a set of torsion classes
\[ \eta:= \{\T_s : s\in[0,1] \text{, } \T_0=\CA, \T_1=\{0\} \text{ and } \T_s \subseteq \T_r \text{ if } r\leq s\}.\]
We denote by $\T(\CA)$ the set of all chains of torsion classes of $\CA$.
\end{definition}

\
Associated to every chain of torsion classes $\eta \in \T(\CA)$ there is a set
$$\CP_\eta = \{ \CP_t : t \in [0,1]\}$$
of full subcategories of $\CA$ where each $\CP_t$ is defined as follows.
Note that in from now on we assume $\bigcap_{s< 0} \T_s=\CA$ and $\bigcup_{s> 1} \T_s = \{0\}$.

\begin{definition}\label{def:chain}
Consider a chain of torsion classes $\eta\in \T(\CA)$.
For every $t \in [0,1]$ we have the subcategory $S_t= \bigcap_{s< t} \T_s\setminus \bigcup_{r> t} \T_r$.
We define $\CP_t$ to be
$$\CP_t:= \{ \text{coker}(\alpha) : \text{where $\alpha: tM_{>t} \rt M$ and $M \in S_t$ }
\},$$
where {$ t{(-)}_{>t}$} is the torsion functor associated  to the torsion class $\bigcup_{s> t} \T_s$ applied to $M$.
We then define the \textit{slicing} $\CP_\eta$ of $\eta$ to be the set $\CP_\eta:= \{ \CP_t : t\in [0,1] \}$.
\end{definition}

We note that this definition is different of the original definition introduced in \cite[Definition~2.8]{T-HN-filt}.
In the following proposition we show that both definitions are  equivalent.

\begin{proposition}
Let $\eta$ be a chain of torsion classes in $\T(\CA)$.
Then for every $t \in [0,1]$ we have that
$$
\CP_t=
\left(\bigcap\limits_{s<t} \T_s\right)  \cap \left(\bigcap\limits_{s>t} \F_s\right)
$$
where $\F_s$ is the torsion-free class such that $(\T_s, \F_s)$ is a torsion pair in $\CA$.
\end{proposition}

\begin{proof}
First recall that, for every $t \in [0,1]$,  by \cite[Proposition~2.6]{T-HN-filt} the tuple \[\left(\bigcup_{s>t} \T_s, \bigcap_{s>t} \F_s\right)\] is a torsion pair.
We now prove the statement by double inclusion.

Fix $t \in [0,1]$ and $X \in\left(\bigcap_{s<t} \T_s\right)  \cap \left(\bigcap_{s>t} \F_s\right)$.
Then, in particular, $X \in\bigcap_{s<t} \T_s$.
Moreover, the short exact sequence of $M$ associated to the torsion pair $\left(\bigcup_{s>t} \T_s, \bigcap_{s>t} \F_s\right)$ is isomorphic to
$$0 \rt 0 \rt X \rt X \rt 0$$
because $X \in\bigcup_{s>t} \T_s$.
Hence $X \in \{ \text{coker}(\alpha) : \text{where $\alpha: tM_{>t} \rt M$ and $M \in S_t$ }\}$.

In the other direction, take $X \in \{ \text{coker}(\alpha) : \text{where $\alpha: tM_{>t} \rt M$ and $M \in S_t$ }\}$.
Then $X$ is the torsion-free quotient of an object $X' \in \bigcap_{s< t} \T_s\setminus \bigcup_{r> t} \T_r$ with respect to the torsion pair $\left(\bigcup_{s>t} \T_s, \bigcap_{s>t} \F_s\right)$.
This implies that $X \in \bigcap_{s>t} \F_s$.
Moreover, since $\bigcap_{s< t} \T_s$ is a torsion class, we know that is close under quotients.
So $X\in \bigcap_{s< t} \T_s$.
Hence $X \in\left(\bigcap_{s<t} \T_s\right)  \cap \left(\bigcap_{s>t} \F_s\right)$.
\end{proof}

One of the main results in \cite{T-HN-filt} shows that every chain of torsion classes $\eta \in \T(\CA)$ induces a Harder-Narasimhan filtration for every non-zero object $M \in \CA$.
The formal statement is the following.

\begin{theorem}\cite[Theorem 2.9]{T-HN-filt}\label{thm:HN-filt}
Let $\CA$ be an abelian category and $\eta \in \T(\CA)$.
Then every object $M\in \CA$ admits a Harder-Narasimhan filtration with respect to $\eta$.
That is a filtration
$$M_0 \subset M_1 \subset \dots \subset M_n$$
such that:
\begin{enumerate}
\item $0 = M_0$ and $M_n=M$;
\item there exists $r_k \in [0,1]$ such that $M_k/M_{k-1} \in \CP_{r_k}$ for all $1\leq k \leq n$;
\item $r_1 > r_2 > \dots > r_n$.
\end{enumerate}
Moreover this filtration is unique up to isomorphism.
\end{theorem}

\subsection{Galois coverings}\label{subsec:Galoiscoverings}

Let $\C$ be a skeletally small Krull-Schmidt $\k $-category, where $\k$ is a field.
Given a group $G$, we say that there is a $G$-action over $\C$ or simply $\C$ is a {\it $G$-category} if there is a  group homomorphism $A: G \lrt {\rm Aut}(\C)$, where ${\rm Aut}(\C)$ denotes the group of $\k$-linear automorphisms of $\C$.
We usually write $A_g$ instead of $A(g)$ and $g X$ for $A_g(X)$, for each $g \in G$ and $X \in \C$.

\begin{definition}\label{def:admissible}
With the above notations, the action of $G$ on $\C$ is called {\emph{admissible}} if $gX\ncong X$ for each indecomposable object $X$ in $\C$ and each $g\neq 1$ in $G$.
\end{definition}

Let $\C$ be a $G$-category.
The {\it orbit category} $\C/G$ of $\C$ by $G$ is a category whose objects are the objects of $\C$ and for every $X,Y \in \C/G$, the morphism set $\C/G(X,Y)$ is given by
\[
\left\{(f_{h, g })_{(g , h)} \in \prod_{(g , h)\in G\times G} \C(g X, h Y)\mid
(f_{h, g })_{(g , h)} \text{ is \textsf{rcf} and } \  f_{g' h,  g'g }= g' (f_{h, g }), \forall g' \in G\right\},
\]
where  \textsf{rcf}  denotes $(f_{h, g })_{(g , h)}$ being row and column finite, that is, for every $g \in G$ there are finitely many $h \in G$ such that $f_{h,g}\neq 0$ and, dually, for every $h'\in G$ there are finitely many $g' \in G$ such that $f_{h',g'}\neq 0$.
For two composable  morphisms $X \st{f}\lrt Y \st{f'}\lrt Z$ in $\C/G$, we define
\[f'f:= \left(\sum_{g' \in G}f'_{h, g'} f_{g', g }\right)_{(g , h) \in G\times G}.\]

There is a canonical functor $P : \C \lrt \C/G$ which is given by $P(X)=X$ and $P(f)= (\delta_{g , h} g  f)_{(g , h)}$, for every $X, Y \in \C$ and for every $f \in \C(X, Y)$.

Recall that a pair $(F, \varphi)$ of a functor $F: \C \lrt \C'$ and  a family $\varphi:=(\varphi_{g })_{g \in G}$ of natural isomorphisms $\varphi_{g}: F \lrt FA_{g }$ is called a $G$-invariant functor if, for every $g  , h \in G$, the following diagram is commutative

\[\begin{tikzcd}
F \ar{r}{\varphi_{g}} \ar{rd}{\varphi_{hg}} & FA_{g} \ar{d}{\varphi_{h} A_{g}}
\\ & FA_{hg}=FA_{h}A_{g}.
\end{tikzcd}\]

The family $\varphi:=(\varphi_{g})_{g  \in G}$ is called an \emph{invariant adjuster of $F$}.

\begin{definition}\label{k-iso}\cite[Definition 1.7]{Asashiba}
Let $F: \C \lrt \C'$ be a $G$-invariant functor. Then $F$ is called a \emph{$G$-precovering functor}, if for every $X,Y \in\C$ the $\k $-homomorphisms
$$ F^{(1)}_{X,Y}: \bigoplus_{g  \in G} \C(g  X, Y) \lrt \C'(FX, FY), \ \ \ \ (f_{g })_{g  \in G} \mapsto \sum_{g  \in G} F(f_{g}) \varphi_{g, X};$$
$$ F^{(2)}_{X,Y} : \bigoplus_{g \in G} \C(X, g Y)\lrt \C'(FX, FY), \ \ \ \ (f_{g})_{g \in G} \mapsto \sum_{g\in G}\varphi_{g^{-1}, gY} F(f_{g}),$$
are isomorphisms. If $F$ is also dense, then it is called a \emph{$G$-covering functor}.
\end{definition}

It is shown in \cite[Proposition 2.6]{Asashiba} that $P: \C \lrt \C/G$ is a $G$-covering functor which is universal among $G$-invariant functors starting from $\C$.

We say that the $\k$-category $\C$ is locally bounded if for each indecomposable $X\in \mathcal{C}$, we have that
$$\sum_{Y \in \ind \mathcal{C}}(\dim_\k(\Hom_\C(X,Y)+\dim_\k(\Hom_\C(Y,X)) < \infty.$$
From now on, assume moreover that $\C$ is a locally bounded $\k$-category.

Let $\C$ be a $G$-category. The $G$-action on $\C$ induces a $G$-action on $\Mod \C$, where $\Mod \C$ denotes the category of  contravariant functors from $\C$ to $\Mod \k$. In fact, for each $g  \in G$, we can define an automorphism  $\overline{A}_{g }:\Mod \C \lrt \Mod \C$ by
\[ \overline{A}_{g }(M)={}^{g } M := M\circ A_{g }^{-1},\]
for all $M \in \Mod \C$.
It follows from the definitions that for every $M \in \C$, ${}^{g }\C(-,M) = \C(g ^{-1}(-),M) \cong \C(-, gM)$.

The canonical functor $P: \C \lrt \C/G$ induces a functor $P^{\bullet}: \Mod (\C/G) \lrt \Mod \C$, given by $P^{\bullet}(M)=M \circ P$ for every $M \in \Mod (\C/G)$. This functor is called the \emph{pull-up of $P$}.
It is well known that $P^{\bullet}$ possesses a left adjoint  $\Pc: \Mod \C \lrt \Mod(\C/G)$, which is called the \emph{push-down functor}.
For an explicit description of this functor see the proof of Theorem 4.3 of \cite{Asashiba}. It follows that the push-down functor $\Pc$ is exact.

A functor $M \in \Mod\C$ is called \emph{finitely presented} if there exists an exact sequence $( - , Y) \lrt ( - , X) \lrt M \lrt 0$ in $\Mod\C$. Let $\mmod\C$ be the full subcategory of $\Mod\C$ consisting of all finitely presented modules. It is known \cite[Theorem~4.3]{Asashiba} that the restriction of the push-down functor to $\mmod\C$ induces a functor
\[\Pc: \mmod \C \lrt \mmod (\C/G)\]
again denoted by $\Pc$, which is a $G$-precovering functor.

The central result relating higher homological algebra and Galois coverings was proved by Darp\"o and Iyama and it reads as follows. Recall that a subcategory $\X$ of $\mmod\C$ is said to be \textit{$G$-equivariant} if ${}^g\X=\X$, for all $g \in G$.

\begin{theorem}\cite[Theorem~2.14]{DarpoIyama}\label{thm:DarpoIyama}
Let $\C$ be a locally bounded Krull-Schmidt $G$-category with an admissible action of $G$ on $\C$ inducing an admissible action on $\mmod \C$.
If $\M$ is a $G$-equivariant full subcategory of $\mmod \C$ such that $P_{\bullet}(\M)$ is functorially finite in $\mmod (\C/G)$, then $\M$ is an $n$-cluster tilting subcategory of $\mmod \C$ if and only if $P_{\bullet}(\M)$ is an $n$-cluster tilting subcategory of $\mmod (\C/G)$.
\end{theorem}

For more details on the covering theory we refer the reader to \cite{Asashiba, BautistaLiu, DarpoIyama}.

\section{Minimal torsion classes containing $n$-torsion classes}\label{sec:ntorsvstors}
Let $\M$ be an $n$-abelian category. By \cite{Kvamme2020}, there exists an abelian category $\CA$ and a fully faithful functor $F: \M \rt \CA$ such that $F(\M)$ is an $n$-cluster-tilting subcategory of $\CA$.
Throughout the section we fix an $n$-abelian category $\M$ and consider it as the $n$-cluster tilting subcategory of the abelian category $\CA$.

For an $n$-torsion class $\U \subseteq \M$, let $T(\U)$ denote the smallest torsion class of $\CA$ containing $\U$ and for any $M \in \M$, denote by  $U^M$  the $n$-torsion object  of $M$ with respect to the $n$-torsion class $\U$.
For a torsion class {$\T$} in $\CA$ and for $M \in \CA$ denote by $tM$ the torsion object of $M$ with respect to {$\T$}.

\begin{lemma}\label{lem:sametors}
Let $\U$ be an $n$-torsion class  in $\M$ and let $M \in \M$.
Then $$tM \cong U^M$$
where $tM$ is the torsion object of  $M$ with respect to $T(\U)$.
In other words, for all $M \in \M$, the fundamental $n$-exact sequence
\begin{equation}
\label{equ:fundamental_n-sequence}
  0 \xrightarrow{} U^M \xrightarrow{ u } M \xrightarrow{} V^1 \xrightarrow{ v^1 } \cdots \xrightarrow{ v^{ n-1 } } V^n \xrightarrow{} 0
\end{equation}
of $M$ with respect to $\U$ is isomorphic to
\begin{equation}
\label{equ:fundamental_n-sequence3}
  0 \xrightarrow{} tM \xrightarrow{ \iota_M } M \xrightarrow{} V^1 \xrightarrow{ v^1 } \cdots \xrightarrow{ v^{ n-1 } } V^n \xrightarrow{} 0.
\end{equation}
\end{lemma}

\begin{proof}
Let $M \in \M$ and take the canonical $n$-exact sequence of $M$ with respect to $\U$
\begin{equation}
  0 \xrightarrow{} U^M \xrightarrow{ u } M \xrightarrow{} V^1 \xrightarrow{v_1 } V^2 \xrightarrow{ } \cdots \xrightarrow{  } V^n \xrightarrow{} 0.
\end{equation}
This induces the following exact sequence in $\CA$
\begin{equation}
  0 \xrightarrow{} \Coker{u}  \xrightarrow{} V^1 \xrightarrow{v^1 } V^2 \xrightarrow{ v^2 } \cdots \xrightarrow{v^{n-1} } V^n \xrightarrow{} 0.
\end{equation}
Applying the functor $\M(U, -)$ with $U \in \U$ we obtain the following exact sequence.
\begin{equation}
  0 \xrightarrow{} \M( U , \Coker{u} ) \xrightarrow{} \M(U, V^1) \xrightarrow{ v^1_* }  \M(U, V^2)
\end{equation}
But from the definition of $n$-torsion class we have that
\begin{equation}\label{eq:n-canonical}
  0 \xrightarrow{} \M(U, V^1) \xrightarrow{ v_*^1 } \cdots \xrightarrow{ v^{n-1}_* } \M(U, V^n) \xrightarrow{} 0
\end{equation}
is exact for every $U \in \U$.
In particular, {the exactness of the sequence} (\ref{eq:n-canonical}) implies that $v_1^* : \M(U, V^1) \rt \M(U, V^2)$ is injective.
Thus $\M(U, \Coker u)=0$ for every $U \in \U$.
This implies that
$$0 \rt U^M \overset{u}{\longrightarrow} M \longrightarrow \Coker u \rt 0$$
is such that $U^M \in T(\U)$ and $\Coker u \in (T(\U))^\perp$, where for a class $\T$ of objects of $\CA$,
$$\T^{\perp}=\{ Y \in \CA \mid \Hom_{\CA}(X,Y)=0, \ {\rm for \ all } \ X \in \T\}.$$
Since the canonical short exact sequence of any object with respect to a torsion pair is unique up-to-isomorphism, we conclude that $tM \cong U^M$.
\end{proof}

As a direct consequence of Lemma \ref{lem:sametors} we obtain the following result.

\begin{proposition}\label{prop:ntorstotors}
Let $\U_1,\U_2$ be two $n$-torsion classes in $\M$. Then $T(\U_1)=T(\U_2)$ if and only if $\U_1 = \U_2$.
\end{proposition}

\begin{proof}
The sufficiency is clear, so we show the necessity.
Let $M \in \U_2 \setminus \U_1$.
Then the torsion object $U_1^M$ of $M$ with respect to $\U_1$ is not isomorphic to $U_2^M=M$.
By Lemma~\ref{lem:sametors} we have that the torsion object $t_1 M$ of $M$ with respect to $T(\U_1)$ is not isomorphic to the torsion object $t_2 M = M$ of $M$ with respect to $T(\U_2)$.
Hence $T(\U_1)$ is different from $T(\U_2)$.
\end{proof}

In recent years there has been a great deal of interest regarding the poset of torsion classes in the module category of an algebra ordered by inclusion.
As a consequence of our previous result, we have the following.

\begin{corollary}\label{cor:poset}
The map $T(-): n\text{-tors}(\M) \rt \text{tors}(\CA)$ from the set of $n$-torsion classes $n\text{-tors}(\M)$ of the $n$-cluster-tilting subcategory $\M$ to the set of torsion classes $\text{tors}(\CA)$ of the abelian category $\CA$ is injective and respects the order given by the inclusion.
\end{corollary}

\begin{proof}
It follows from Proposition \ref{prop:ntorstotors} that if $\U_1$ and $\U_2$ are two distinct $n$-torsion classes in $\M$, then $T(\U_1)$ and $T(\U_2)$ are two distinct torsion classes in $\CA$.
The fact that $T(\U_1) \subset T(\U_2)$ if $\U_1 \subset \U_2$ is immediate.
\end{proof}

\begin{lemma}\label{lem:T(U)}
Let $\U$ be an $n$-torsion class of $\M \subset\CA$.
If $\T = T(\U)$ is the minimal torsion class of $\CA$ containing $\U$, then $\U = t \M$, where $t \M=\{tM : M\in \M\}$.
In particular we have $\T = T(t \M)$.
\end{lemma}

\begin{proof}
The fact that $\U \supseteq \{tM : M \in \M\}$ follows directly from Lemma \ref{lem:sametors}.
Let $U \in \U \subset \M$.
Then $U \in \T = T(\U)$.
Hence $tU= U$.
Thus $U \in \{tM : M\in \M\}$.
\end{proof}

\begin{lemma}\label{lem:iii}
Let $\M \subseteq \CA$ be an $n$-cluster tilting subcategory and $\T \subseteq \CA$ be a torsion class satisfying $t\M \subseteq \M$. Then $t\M$ is an $n$-torsion class of $\M$ if and only if $\Ext_{ \CA }^{ n-1 }( X, Y) = 0$ for all $X \in t\M$ and all $Y \in f\M$, where $f\M=\{\Coker(tM \hookrightarrow M) \mid M \in \M\}$.
\end{lemma}

\begin{proof}
Assume first that $t\M \subseteq \M$ is an $n$-torsion class. Let $X \in t\M$ and $Y \in f\M$ be arbitrary elements. So $Y=fM$, for some $M \in \M$ and there exists a short exact sequence
\begin{equation}
\label{equ:fundamental_sequence2}
  0 \xrightarrow{} tM \xrightarrow{ \iota_M } M \xrightarrow{} fM \xrightarrow{} 0,
\end{equation}
where the monomorphism $tM \st{\iota_M}{\hookrightarrow} M$ sits in the canonical $n$-exact sequence
\begin{equation}
\label{equ:lemma_sequence}
 0 \xrightarrow{} tM \xrightarrow{ \iota_M } M \xrightarrow{} V^1 \xrightarrow{ v^1 } \cdots \xrightarrow{ v^{ n-1 } } V^n \xrightarrow{} 0,
\end{equation}
in which the sequence
\begin{equation}
\label{equ:resolution}
  0  \xrightarrow{} V^1 \xrightarrow{ v^1 } \cdots \xrightarrow{ v^{ n-1 } } V^n \xrightarrow{} 0
\end{equation}
is $t\M$-exact.

The sequence \eqref{equ:fundamental_sequence2} induces a long exact sequence which contains
\begin{equation}\label{equ:extseq}
    \Ext_{ \CA }^{ n-1 }(X, M ) \xrightarrow{}   \Ext_{ \CA }^{ n-1 }(X, fM ) \xrightarrow{}   \Ext_{ \CA }^n(X, tM ) \xrightarrow{ \Ext_{ \CA }^n(X, \iota_M ) } \Ext_{ \CA }^n(X, M).
\end{equation}

Since $X \in t\M \subseteq \M$ while $\M$ is $n$-cluster tilting, the first term is zero and so the sequence reads as follows
\begin{equation}
\label{equ:new sequence}
     0 \xrightarrow{} \Ext_{\CA}^{ n-1 }(X, fM) \xrightarrow{} \Ext_{ \CA }^n(X, tM) \xrightarrow{ \Ext_{ \CA }^n(X, \iota_M) } \Ext_{ \CA }^n(X, M).
\end{equation}

Moreover, $X \in \M$ implies that the sequence \eqref{equ:lemma_sequence} induces a long exact $\Hom$-$\Ext^n$-sequence which contains the following, see \cite[Prop.\ 2.2]{Jasso2019},
\begin{equation}\label{equ:Jasso}
    {\CA}(X, V^{ n-1 } ) \xrightarrow{v^{n-1}_*} { \CA }(X, V^n ) \xrightarrow{} \Ext_{ \CA }^n(X, tM ) \xrightarrow{ \Ext_{ \CA }^n(X, \iota_M) } \Ext_{ \CA }^n(X, M).
\end{equation}
Since the sequence \eqref{equ:resolution} is $t\M$-exact, we deduce that $v^{n-1}_*$ is surjective and so the morphism $\Ext_{ \CA }^n(X, \iota_M)$ is injective.  This in view of the sequence \eqref{equ:new sequence} implies that $\Ext_{\CA}^{n-1}(X, Y)\;=\;0$

For the converse, assume that $\Ext_{ \CA }^{ n-1 }( X, Y) = 0$, for all $X \in t\M$ and all $Y \in f\M$. Consider $M \in \M$. Since $t\M \subseteq \M,$ we have an $n$-exact sequence

\begin{equation}
\label{equ:n-exact}
 0 \xrightarrow{} tM \xrightarrow{ \iota_M } M \xrightarrow{} V^1 \xrightarrow{ v^1 } \cdots \xrightarrow{ v^{ n-1 } } V^n \xrightarrow{} 0.
\end{equation}
in $\M$, where $\Coker(tM \st{\iota_M}{\hookrightarrow} M)=fM \in f\M$. To conclude the result we should show that the sequence
\begin{equation*}
  0 \xrightarrow{} V^1 \xrightarrow{ v^1 } \cdots \xrightarrow{ v^{ n-1 } } V^n \xrightarrow{} 0.
\end{equation*}
is $t\M$-exact. Since \eqref{equ:n-exact} is an $n$-exact sequence, we just need to show that for every $X \in t\M$, the induced morphism $\CA(X, v^1)$ is an injection and  the induced morphism $\CA(X, v^{n-1})$ is a surjection. But $\CA(X, v^1)$ is injective, because $\CA(t\M,  fM)=0$. To see that $\CA(X, v^{n-1})$ is surjective, it is enough to show that the morphism $\Ext_{ \CA }^n(X, \iota_M)$ in the sequence \eqref{equ:Jasso} is injective. But this follows from the sequence \eqref{equ:extseq} in view of the fact that by assumption $\Ext_{ \CA }^{ n-1 }(X, fM )=0$.
\end{proof}

We are now in place to give a characterisation of the torsion classes $\T$ in $\CA$ which are of the form $\T=T(\U)$ for some $n$-torsion class $\U$ of $\M$.

\bigskip
\begin{theorem}\label{thm:torsfromntors}
Let $\M \subseteq \CA$ be an $n$-cluster tilting subcategory and $\T$ be a torsion class of $\CA$.
Then $\T$ is of the form $\T=T(\U)$ for some $n$-torsion class $\U$ of $\M$ if and only if the following holds:
\begin{enumerate}
    \item $t\M \subseteq \M$;
    \item $\T=T(t\M)$;
    \item ${\rm Ext}^{n-1}_{\CA} (X, Y) = 0$ for all $X \in t\M$ and $Y \in f\M$.
\end{enumerate}
Moreover, in this case $\U= \T \cap \M=\{ tM : M \in \M \}$.
\end{theorem}

\begin{proof}
\textit{Necessity}.
Suppose that $\T = T( \U )$ for an $n$-torsion class $\U \subseteq \M$.
We must show that $\T$ has all three characteristics as in the statement.
Parts 1. and 2. follow from Lemma \ref{lem:T(U)}.
In particular we have $\U= t\M \subseteq \M$, so Lemma \ref{lem:iii} applies to complete the proof of this part.

\smallskip
\textit{Sufficiency}.
Suppose that $\T$ is a torsion class in $\CA$ satisfying 1.-3.
We must show $\T = \T( \U )$ for an $n$-torsion class $\U \subseteq \M$, and in view of 2. this holds if $t\M$ is an $n$-torsion class in $\M$.
But $t\M \subseteq \M$ holds by 1, and so by 3. in view of Lemma \ref{lem:iii}, the sequence
\eqref{equ:lemma_sequence} is
is a fundamental $n$-exact sequence for $M$ with respect to $t\M$.
\smallskip

Now we show the moreover part of the statement.
It is already proven in Lemma~\ref{lem:T(U)} that $\U=\{ tM : M \in \M \}$.
So we only need to prove that $\U= \T \cap \M$.
We do it by double inclusion.
The fact that $\U \subset \M \cap \T$ follows immediately from the fact that $\T=T(\U)$.
Now, if $X \in \M \cap \T$ we have that $X = tX$.
Then $X \in \U$ by Lemma~\ref{lem:sametors}.
\end{proof}

\begin{example}
Let $A$ as in Example~\ref{ex:running} and consider the torsion classes $\T_1 = \add \{ \rep{2\\3} \oplus \rep{2} \oplus \rep{1\\2} \oplus \rep{1}\}$ and $\T_2 = \add \{\rep{2} \oplus \rep{1\\2}\oplus \rep{1} \}$ and $\T_3=\add \{\rep{1} \oplus \rep{3}\}$ in $\mmod A$.

First, we have that $\T_1$ verifies 1.--3. in the previous theorem and hence is the minimal torsion class containing the $2$-torsion class $\add\left\{\rep{2\\3} \oplus \rep{1\\2} \oplus \rep{1}\right\}$. On the other hand, neither $\T_2$ nor $\T_3$ are minimal classes containing a $2$-torsion class.

In the case of $\T_2$ one can see that 2. fails, that is,  $\T_2 \neq t\M$.
Indeed, $$\T_2 =\add \left\{\rep{2} \oplus \rep{1\\2}\oplus \rep{1} \right\} \neq  \add \left\{ \rep{1\\2}\oplus \rep{1}\right\}= t \M.$$
Nevertheless, we point out that $\T_2 \cap \M =  \add \left\{ \rep{1\\2}\oplus \rep{1}\right\}$ is a $2$-torsion class in $\M$.

Finally, in the case of $\T_3$ the problem is that 3. is not satisfied since the canonical exact sequence of $\rep{2\\3}$ with respect to the torsion pair $(\T_3, \F_3)$ is
$$0 \longrightarrow \Rep{3} \longrightarrow \Rep{2\\3} \longrightarrow \Rep{2} \longrightarrow 0 $$
and $\Ext^{1}_A(\rep{1}, \rep{2}) \not\cong 0$.
In fact, the minimal $2$-torsion class in $\M$ containing $\add \{\rep{1} \oplus \rep{3}\}$ is the whole of $\M$.
\end{example}

\section{Harder-Narasimhan filtrations in $n$-abelian categories}\label{sec:higherHNfilt}
We start this section by introducing chains of $n$-torsion classes.

\begin{definition}\label{def:n-chain}
A chain of $n$-torsion classes $\delta$ in an $n$-abelian category $\M$ is a set of $n$-torsion classes
$$\delta:= \{\U_s : s \in [0,1], \U_0=\M, \U_1=\{0\} \text{ and } \U_s \subseteq \U_r \text{ if } r\leq s\}.$$
We denote by $\T(\M)$ the set of all chains of $n$-torsion classes in $\M$.
\end{definition}
In this section we show that every chain of $n$-torsion classes $\delta$ in $\T(\M)$ induces an $n$-Harder-Narasimhan filtration for every non-zero object $M \in \M$.
We first need some preliminary results.

\begin{lemma}\label{lem:subtorsion}
Let $\U_1 \subset \U_2$ be two $n$-torsion classes in an $n$-abelian category $\M$ and let $M$ be an object of $\M$.
Take the $n$-torsion subobjects $U^M_1$, $U^M_2$ of $M$ with respect to $\U_1$ and $\U_2$, respectively.
Then the following hold:
\begin{enumerate}
    \item $U^M_1$ is a subobject of $U^M_2$.
    \item The torsion object $U^{U^M_2}_1$ of $U^M_2$ with respect to $\U_1$ is isomorphic to $U^M_1$.
\end{enumerate}
\end{lemma}

\begin{proof}
1. This follows directly from the definition of $n$-torsion classes and the fact that $U^M_1$ is an object of $\U_2$.

\noindent
2. Since $U^M_2$ is a subobject of $M$, we have that $U_1^{U_2^M}$ is a subobject of $M$ which belongs to $\U_1$.
Then, by the universal properties of $n$-torsion objects we have that $U_1^{U_2^M}$ is a subobject of $U_1^{M}$.
On the other hand, we have by $1.$ that $U_1^M$ is a subobject of $U_2^M$.
Hence, the universal properties of torsion objects imply that $U_1^M$ is a subobject of $U_1^{U_2^M}$.
We can then conclude that $U_1^M$ is isomorphic to $U_1^{U_2^M}$.
\end{proof}

\begin{proposition}\label{prop:limits}
Let $\M$ be an $n$-abelian category and let $\delta$ be a chain of $n$-torsion classes in $\M$.
Then $\bigcup_{r> s} \U_r$ is an $n$-torsion class in $\M$ for all $s\in [0,1)$ and $\bigcap_{t< s} \U_t$ is an $n$-torsion class in $\M$ for all $s \in (0,1]$.
\end{proposition}

\begin{proof}
Let $\delta\in\T(\M)$ and $M \in \M$.
We first show that $\bigcup_{r> s} \U_r$ is an $n$-torsion class for all $s\in [0,1)$.
Take for every $t > s$ the $n$-torsion subobject $U^M_t$ of $M$ with respect to $\U_t$.
Then Lemma~\ref{lem:subtorsion} implies that we have an ascending chain of subobjects of $M$ as follows
$$0 = U_1^M \subset \dots \subset U_t^M \subset \dots \subset M.$$
Recall that $\CA$ is a length category, in particular $\CA$ is an Artinian category.
This implies that the above ascending chain of subobjects of $M$ stabilises.
In other words, this implies the existence of a $t_M > s$ such that $U^M_{t_M} = U^M_t$ for all $s < t <t_M$.

Given that $U^M_{t_M}$ is a subobject of $M$, there is a monomorphism $\alpha: U^M_{t_M} \rt M$.
Then we obtain an $n$-exact sequence in $\M$
\begin{equation}
  0 \xrightarrow{} U^M_{t_M} \xrightarrow{\alpha} M \xrightarrow{} V^1 \xrightarrow{ v^1 } \cdots \xrightarrow{ v^{ n-1 } } V^n \xrightarrow{} 0
\end{equation}
by taking the $n$-cokernel of $\alpha$.
We claim that
\begin{equation}
    \label{eq:exactlimits}
  0 \xrightarrow{} \M(X,V^1) \xrightarrow{ } \M(X,V^2) \xrightarrow{  } \cdots \xrightarrow{  } \M(X,V^n) \xrightarrow{} 0.
\end{equation}
is exact for all $X \in \bigcup_{r> s} \U_r$.
Indeed, for each $X \in \bigcup_{r> s} \U_r$ there exists a real number $t_1 \in (s, 1]$ such that $X \in \U_{t_1}$.
If $t_M \leq t_1$, we have that $X \in \U^M_{t_M}$.
Then (\ref{eq:exactlimits}) is exact because $U^M_{t_M}$ is the torsion object of $M$ in $\U_{t_M}$.
Otherwise, let $s < t_1 < t_M$.
Then we have that $U_{t_1}^M\cong U^M_{t_M}$ by construction and (\ref{eq:exactlimits}) is exact for all $X \in \bigcup_{r> s} \U_r$.
Hence, we have shown that for each $M \in \M$ there exists a subobject $U_{>s}^M:=U^M_{t_M}$ in $\bigcup_{r> s} \U_r$ such that (\ref{eq:exactlimits}) is exact for all $X \in \bigcup_{r> s} \U_r$.
Thus $\bigcup_{r> s} \U_r$ is an $n$-torsion class.

To show that $\bigcap_{t< s} \U_t$ is an $n$-torsion class in $\M$ for all $s\in (0,1]$, we start by noting that for all $M\in \M$ there exists a descending chain of subobjects
$$0 \subset \dots \subset U^M_t \subset \dots \subset U^M_0=M$$
where $t\in [0,s)$.
Since $\CA$ is a length category we have that $\CA$ is Noetherian, which implies the existence of $t_M\in [0,s)$ such that $U^M_{t_M} = U^M_t$ for all $t_M < t < s$.
Let $\alpha: U^M_{t_M} \rt M$ be a monomomorphism and consider the exact sequence
\begin{equation}
    \label{eq:exactlimits2}
  0 \xrightarrow{} U^M_{t_M} \xrightarrow{\alpha} M \xrightarrow{} V^1 \xrightarrow{ v^1 } \cdots \xrightarrow{ v^{ n-1 } } V^n \xrightarrow{} 0
\end{equation}
in $\M$ that comes from taking the $n$-cokernel of $\alpha$.
Then the sequence
\begin{equation}
\label{eq:exactlimits3}
  0 \xrightarrow{} \M(X,V^1) \xrightarrow{ } \M(X,V^2) \xrightarrow{  } \cdots \xrightarrow{  } \M(X,V^n) \xrightarrow{} 0
\end{equation}
is exact for all $X \in \bigcap_{t< s} \U_t \subset \U_{t_M}$ because $U^M_{t_M}$ is the torsion object of $M$ with respect to $\bigcap_{t< s} \U_t$.
Hence for every $M\in \M$ there exists a subobject $U^M_{<s}:=U^M_{t_M}$ of $M$ such that $U^M_{<s}\in \bigcap_{t< s} \U_t$ and (\ref{eq:exactlimits3}) is exact for every $X \in \bigcap_{t< s} \U_t$.
In other words, $\bigcap_{t< s} \U_t$ is an $n$-torsion class.
\end{proof}

Given an object $M$ in $\M$,  $\delta$ a chain of $n$-torsion classes in $\M$ and $s \in (0,1)$ we denote by $U^M_{>s}$ and $U^M_{<s}$ the torsion object of $M$ with respect to $\bigcup_{r> s} \U_r$ and $\bigcap_{t< s} \U_t$, respectively.

The above results allow us to define the notion of slicing for chains of $n$-torsion classes which will enable us to show the existence of Harder-Narasimhan filtrations from chains of $n$-torsion classes.

\begin{definition}\label{def:n-slicing}
Let $\delta$ be a chain of $n$-torsion classes.
For every $t \in [0,1]$ we have the subcategory $S_t= \bigcap_{s< t} \U_s\setminus \bigcup_{r> t} \U_r$.
We define $\Q_t$ to be
$$\Q_t:= \{ n\text{-coker}(\alpha) : \text{where $\alpha: U^M_{>t} \rt M$ and $M \in S_t$ }
\}.$$
Moreover the \textit{slicing} $\Q_\delta$ of $\delta$ is the set $\Q_\delta:= \{ \Q_t : t\in [0,1] \}$.
\end{definition}

\begin{remark}
Note that for a given chain of $n$-torsion classes $\delta$,  $\Q_t$ might be empty for some $t\in [0,1]$.
\end{remark}

\begin{theorem}\label{thm:nHNfilt}
Let $\M$ be an $n$-abelian category and let $\delta$ be a chain of $n$-torsion classes in $\M$.
Then $\delta$ induces an $n$-Harder-Narasimhan filtration for every $M \in \M$.
That is a filtration
$$0 =M_0 \subsetneq M_1 \subsetneq \dots \subsetneq M_{r-1} \subsetneq M_r = M$$
such that there exists a finite ordered set $s_1 > s_2 > \dots > s_{r}$ such that $s_i\in [0,1]$ and the $n$-cokernel of the inclusion $M_{k-1} \rt M_k$ is in $\Q_{s_k}$ for every $1\leq k\leq r$.
Moreover this filtration is unique up-to-isomorphism.

In particular, we have that $M_{i-1}$ is the torsion subobject of $M_{i}$ with respect to $ \bigcup_{r > s_{i}} \U_r$ for all $1 \leq i \leq r$.
\end{theorem}

\begin{proof}
Given $M \in \M$, we begin by showing the existence of a filtration with the desired properties.
For this we need to show the existence of $s_r\in[0,1]$ such that $M \in \bigcap_{s< s_r} \U_s\setminus~\bigcup_{t> s_r}~\U_{t}$.
Clearly $M \in \T_{0}=\M$ and $M \not\in \T_1=\{0\}$.
Moreover either $M \in \U_s$ or $M\not \in \U_s$ for all $s\in [0,1]$.
Hence
$$s_r= \inf \{ t\in [0,1] : M\not \in \U_t\} = \sup \{ s\in [0,1] : M \in \U_s\}$$
is well-defined and uniquely determined by $M$.
Now, consider the $n$-torsion subobject $U^M_{>s_{r}}$ of $M$ with respect to the $n$-torsion class $\bigcup_{t > s_{r}} \U_t$.
Note that $U^M_{>s_{r}}$ is a proper subobject of $M$ since $M \not\in \bigcup_{t > s_{r}} \U_t$.
Hence the $n$-cokernel of the natural inclusion $U^M_{>s_{t}} \rt M$ is in $\Q_{s_r}$.

Set $M_r :=M$ and  $M_{r-1} := U^M_{>s_{r}}$.
Applying the above  argument to $M_{r-1}$, there is a unique $s_{r-1}\in[0,1]$ such that $M_{r-1} \in \bigcap_{s< s_{r-1}} \U_s\setminus \bigcup_{t> s_{r-1}} \U_{t}$.
Moreover,
$$s_{r-1}= \inf \{ t\in [0,1] : M_{r-1}\not \in \U_t\} = \sup \{ s\in [0,1] : M_{r-1}\in \U_s\}.$$
Note that $M_{r-1} \in \bigcup_{t > s_{r}} \U_t \subset \U_{s_r}$.
In particular, this implies that $s_{r-1}> s_r$.

Applying this process inductively, we get an ascending sequence $s_{r} < s_{r-1} < \dots $ corresponding to a descending chain of subobjects of $M$
$$\dots \subset M_{r-i} \subset M_{r-i+1} \subset \dots \subset M_r=M.$$
Recall that, by Theorem~\ref{thm:embedding} $\M$ is a full subcategory of an abelian category $\CA$, which is assumed to be a length category.
Hence $M$ is of finite length in $\CA$.
Thus there is a positive integer $k$ that $M_{r-k}=0$.
Without loss of generality we can suppose that $k=r$.
This shows the existence of  a filtration with the desired properties.

We now show the uniqueness of this filtration.
Suppose that there exists a second filtration
$$0 =M'_0 \subsetneq M'_1 \subsetneq \dots \subsetneq M'_{t-1} \subsetneq M'_t = M$$
as in the statement of the theorem.
Then $M'_{t-1}$ is the torsion object $U^M_{>s_t}$ of $M$ with respect to the $n$-torsion class  $\bigcup_{x > s_t} \U_x$ for some $s_t$ in $[0,1]$ such that $M \in \bigcap_{s< s_t} \U_s $ $\setminus~\bigcup_{t> s_t}~\U_{t}$.
However, we have shown that there exists a unique $s_t \in [0,1]$ such that~$M \in \bigcap_{s< s_t} \U_s\setminus~\bigcup_{t> s_t}~\U_{t}$.
This implies that $s_t=s_r$.
Moreover, we have that $M_{r-1} \cong M'_{t-1}$ since the torsion objects are unique up to isomorphism.

Repeating this process we show that $s_{t-i}=s_{r-i}$ and that $M_{r-i} \cong M'_{t-i}$ for all positive integer $i$.
In particular we have that $0 \neq M_1 \cong M'_{r-t+1}$ and $0=M_0 \cong M'_{r-t}$, implying that $r=t$ and the proof is complete.
\end{proof}

\begin{example}\label{ex:n-HN}
Consider once again the algebra $A$ with $2$-cluster tilting subcategory $\M$ as in Example~\ref{ex:running}.
Then
$$\delta=
\begin{cases}
\U_t = \M  & \text{ if $t \in [0, 1/3)$}\\
\U_t = \add\{\rep{3}\}  & \text{ if $t \in [1/3, 2/3)$}\\
\U_t = \add\{0\}  & \text{ if $t \in [2/3, 1]$}
\end{cases}
$$
is a chain of $2$-torsion classes in $\M$.
The following is a complete list of the $2$-Harder-Narasimhan filtrations induced by $\delta$ in the indecomposable objects of $\M$.

\begin{center}
\begin{tabular}{|| c  | c | c ||}
\hline
Object & $2$-Harder-Narasimhan filtration & $\{s_1, \dots, s_m\}$\\
\hline
$\Rep{1}$	& $0 \subset \Rep{1}$   & $s_1=1/3$\\
\hline
$\Rep{1\\2}$	& $0 \subset \Rep{1\\2}$& $s_1=1/3$\\
\hline
$\Rep{2\\3}$	& $0 \subset \Rep{3} \subset \Rep{2\\3}$ & $s_1 = 2/3$,  $s_2=1/3$\\
\hline
$\Rep{3}$	& $0 \subset \Rep{3}$ & $s_1 = 2/3$ \\
\hline
\end{tabular}
\end{center}
\end{example}

\section{Embedding of $n$-Harder-Narasimhan filtrations}\label{sec:nHNvsHN}

In Section \ref{sec:ntorsvstors} we have shown that the map $T(-): n\text{-tors}(\M) \rt \text{tors}(\CA)$ embeds the poset of $n$-torsion classes $n\text{-tors}(\M)$ in an $n$-cluster tilting subcategory $\M$ of an abelian category $\CA$ into the poset tors$(\CA)$ of torsion classes in $\CA$.
This implies, in particular, that every chain of $n$-torsion classes $\delta$ in $\M$ induces naturally a chain of torsion classes $T(\delta)$ in $\CA$ by setting
$$T(\delta) := \{T(\U_s) : \U_s\in\delta \text{ for all $s \in [0,1]$}\}. $$

In order to construct ($n$-)Harder-Narasimhan filtrations and show that they are unique, we use infinite unions and intersections of ($n$-)torsion classes.
We now show that the map $T(-): n\text{-tors}(\M) \rt \text{tors}(\CA)$ commutes with infinite unions and intersections.

\begin{proposition}\label{prop:Tcommutes}
Let $\delta = \{ \U_s : s\in[0,1] \}$ be a chain of $n$-torsion classes in $\M$.
Then
$$T\left( \bigcup\limits_{r > s} \U_r\right)=  \bigcup\limits_{r > s} T(\U_r) \quad \text{ and } \quad T\left( \bigcap\limits_{r < s} \U_r\right)=  \bigcap\limits_{r<s} T(\U_r)$$
for all $s \in [0,1]$.
Moreover, if $F(\U)$ is the torsion free class in $\CA$ such that $(T(\U), F(\U))$ is a torsion pair in $\CA$, then
$$\left(\bigcup\limits_{r > s} T(\U_r) , \bigcap\limits_{r > s} F(\U_r)  \right) \text{ and } \left(\bigcap\limits_{r < s} T(\U_r) , \bigcup\limits_{r < s} F(\U_r)  \right)$$
are torsion pairs in $\CA$ for all $s \in [0,1]$.
\end{proposition}

\begin{proof}
We prove that $T\left( \bigcup\limits_{r > s} \U_r\right)=  \bigcup\limits_{r > s} T(\U_r)$.

Clearly $\U_r \subset T(\U_r)$ for all $r > s$.
Then $\bigcup\limits_{r > s} \U_r \subset \bigcup\limits_{r > s} T(\U_r)$.
We can then apply $T(-)$ to both sets to obtain $T\left(\bigcup\limits_{r > s} \U_r\right) \subset T\left( \bigcup\limits_{r > s} T(\U_r) \right)$.
Now, we have that $\bigcup\limits_{r > s} T(\U_r)$ is a torsion class by \cite[Proposition 2.3]{T-HN-filt}.
So, $T\left(\bigcup\limits_{r > s} \U_r\right) \subset T\left( \bigcup\limits_{r > s} T(\U_r) \right) =  \bigcup\limits_{r > s} T(\U_r)$.

In the other direction, recall that $T(X) = \Filt (\Fac(X))$ (cf. \cite{DIJ}).
Then we have the following inclusions.
\begin{align*}
    \U_r &\subset \bigcup\limits_{r > s} \U_r & & \text{ for all $r > s$} \\
    \Filt(\Fac(\U_r)) &\subset \Filt\left(\Fac\left(\bigcup\limits_{r > s} \U_r \right)\right) & & \text{ for all $r > s$} \\
    T(\U_r) &\subset T\left(\bigcup\limits_{r > s} \U_r \right) & & \text{ for all $r > s$} \\
    \bigcup_{r > s} T(\U_r) &\subset T\left(\bigcup\limits_{r > s} \U_r \right) &
\end{align*}
The moreover part of the statement follows directly from \cite[Proposition\;2.7]{T-HN-filt}.
\end{proof}

By now, we have seen that for every non-zero object $M \in \M \subset \CA$ we have the $n$-Harder-Narasimhan filtration induced by $\delta$ given by Theorem \ref{thm:nHNfilt} and the Harder-Narasimhan filtration induced by $T(\delta)$ given by Theorem \ref{thm:HN-filt}.
We show that both filtrations coincide.

\begin{theorem}\label{thm:comparison}
Let $\delta$ be a chain of $n$-torsion classes in $\M$, $T(\delta)$ be the chain of torsion classes in $\CA$ induced by $\delta$ and $M$ be an object in $\M$.
Consider the $n$-Harder Narasimhan filtration
$$0 =M_0 \subsetneq M_1 \subsetneq \dots M_{t-1} \subsetneq M_t = M$$
of $M$ induced by $\delta$ and the Harder-Narasimhan filtration
$$0 =N_0 \subsetneq N_1 \subsetneq \dots M_{t'-1} \subsetneq M_{t'} = M$$
of $M$ induced by $T(\delta)$.
Then we have that $t=t'$ and $M_i \cong N_i$ for all $1 \leq i \leq t$.
\end{theorem}

\begin{proof}
Let $M$ be an object of $\M$ and $\delta$ a chain of $n$-torsion classes.
Consider the $n$-Harder Narasimhan filtration
$$0 =M_0 \subsetneq M_1 \subsetneq \dots M_{t-1} \subsetneq M_t = M$$
of $M$ induced by $\delta$.
By Theorem~\ref{thm:nHNfilt}, we have that $M_{i-1}$ is the $n$-torsion subobject of $M_i$ with respect to the $n$-torsion class $\bigcup\limits_{r > s_{i}} \U_r$.
Applying Lemma~\ref{lem:sametors} and Proposition~\ref{prop:Tcommutes} we obtain that $M_{i-1}$ is the torsion object of $M_i$ with respect of $\bigcup\limits_{r > s_{i}} T(\U_r)$.
In other words,
$$0 \lrt M_{i-1} \stackrel{\alpha}{\lrt} M_i \lrt \Coker (\alpha) \lrt 0$$
is the canonical short exact sequence of $M_i$ with respect to the torsion pair
$$\left(\bigcup\limits_{r > s_{i}} T(\U_r) , \bigcap\limits_{r > s_{i}} F(\U_r)  \right),$$
where $F(\U_r)$ is the torsion free class in $\CA$ such that $(T(\U_r), F(\U_r))$ is a torsion pair (see \cite[Proposition 2.7]{T-HN-filt}).
Thus $\Coker (\alpha)$ belongs to $\bigcap\limits_{r > s_{i}} F(\U_r)$.

On the other hand, it follows from Theorem~\ref{thm:nHNfilt} that $M_i \in \U_s$ for all $s < s_i$.
Then $M_i \in T(\U_s)$ for all $s < s_i$.
Hence $M_i \in \bigcap\limits_{s<s_i} T(\U_s)$.
So, $\Coker(\alpha)$ belongs to $\bigcap\limits_{s<s_i} T(\U_s)$, since $\bigcap\limits_{s<s_i} T(\U_s)$ is a torsion class by \cite[Proposition 2.7]{T-HN-filt}.

We can then conclude that
$\Coker(\alpha)\in \CP_{s_i} = \bigcap\limits_{s<s_i} \T(\U_s)  \cap \bigcap\limits_{s>s_i} F(\U_s)$
for all $1 \leq i \leq t$.

This means that
$$0 =M_0 \subsetneq M_1 \subsetneq \dots M_{t-1} \subsetneq M_t = M$$
is a filtration of $M$ such that:
\begin{enumerate}
\item $0 = M_0$ and $M_n=M$;
\item there exists $s_k \in [0,1]$ such that $M_k/M_{k-1} \in \CP_{s_k},$ for all $1\leq k \leq t$;
\item $s_1 > s_2 > \dots > s_t$.
\end{enumerate}
Then this is the Harder-Narasimhan filtration of $M$ with respect to the chain of torsion classes $T(\delta)$.
\end{proof}

As a consequence of the last theorem, we have the following.

\begin{corollary}
Let $\delta$ be a chain of $n$-torsion classes of $\M \subset \CA$ and consider the family of subcategories $\{\CP_s : s\in[0,1]\}$ be the slicing associated to $T(\delta)$ as defined in Definition \ref{def:chain}.
If there exists a non-zero object in $\M$ then there exists a $s_1 \in [0,1]$ such that $\CP_{s_1} \cap \M$ contains a non-zero object.
\end{corollary}

\begin{proof}
Let $\delta$ be a chain of $n$-torsion classes in $\M$ and $M$ be a non-zero object of $\M$.
Then, Theorem \ref{thm:nHNfilt} gives us the $n$-Harder-Narasimhan filtration of $M$
$$0 =M_0 \subsetneq M_1 \subsetneq \dots M_{t-1} \subsetneq M_t = M.$$
Moreover, Theorem \ref{thm:comparison} implies that $M_i / M_{i-1} \in \CP_{s_i}$ for all $1 \leq i \leq t$.
In particular, $M_1/M_0 \cong M_1 \in \CP_{s_1}$.
\end{proof}

\begin{example}
To finish this section, let $A$ and $M$ as in Example~\ref{ex:running} and let
$$\delta=
\begin{cases}
\U_t = \M  & \text{ if $t \in [0, 1/3)$}\\
\U_t = \add\left\{\rep{1}\oplus\rep{1\\2}\oplus\rep{2\\3}\right\}  & \text{ if $t \in [1/3, 2/3)$}\\
\U_t = \add\{0\}  & \text{ if $t \in [2/3, 1]$}
\end{cases}
$$
be a chain of $2$-torsion classes in $\M$.
Is easy to see that
$$T(\delta)=
\begin{cases}
T(\U_t)= \mmod A & \text{ if $t \in [0, 1/3)$}\\
T(\U_t) = \add\left\{\rep{1}\oplus\rep{1\\2}\oplus\rep{2\\3}\oplus \rep{2}\right\}  & \text{ if $t \in [1/3, 2/3)$}\\
T(\U_t) = \add\{0\}  & \text{ if $t \in [2/3, 1]$}
\end{cases}
$$
is the chain of torsion classes $T(\delta)$ induced by $\delta$ in $\mmod A$.
\end{example}

\section{Galois coverings of $n$-torsion classes}

Assume that $\C$ is a locally bounded Krull-Schmidt $\k$-category, where $\k$ is a field, with an admissible $G$-action on $\C$ inducing an admissible action on $\mmod \C$.
Then, as in Subsection~\ref{subsec:Galoiscoverings}, the functor $P: \C \rt \C/G$ induces a functor $\Pc: \mmod \C \lrt \mmod (\C/G)$ which is a $G$-precovering map.
Moreover, we know by Theorem~\ref{thm:DarpoIyama} that, under these assumptions, a $G$-equivariant subcategory $\M$ of $\mmod \C$ such that $\Pc(\M)$ is functorially finite in $\mmod(\C/G)$ is an $n$-cluster-tilting subcategory of $\mmod\C$ if and only if $\Pc(\M)$ is an $n$-cluster-tilting subcategory of $\mmod \C/G$. For more details see Subsection~\ref{subsec:Galoiscoverings}.

We start this section by showing that $G$-equivariant $n$-torsion classes behave well under the push-down functor $\Pc: \mmod \C \lrt \mmod (\C/G)$.

\begin{theorem}\label{thm:cov-main}
Let  $\C$ be a locally bounded Krull-Schmidt $\k$-category with an admissible action of a group $G$ on $\C$ inducing an admissible action on $\mmod \C$.
Suppose that $\M$ is an $n$-cluster-tilting  $G$-equivariant full subcategory of $\mmod \C$ such that $P_{\bullet}(\M)$ is functorially finite in $\mmod\C/G$. Let $\U$ be a $G$-equivariant full subcategory of $\M$. If $\U$ is an $n$-torsion class of $\M$ then $P_{\bullet}(\U)$ is an $n$-torsion class of $P_{\bullet}(\M)$.
\end{theorem}

\begin{proof}
First, note that by Theorem~\ref{thm:DarpoIyama}, $P_{\bullet}(\M)$ is an $n$-cluster-tilting subcategory of $\mmod\C/G$.
Let $\U$ be an $n$-torsion class of $\M$ and $P_{\bullet}(M)$ be an object in $P_{\bullet}(\M)$.
By definition, there is an $n$-exact sequence
\begin{equation}\label{seq}
 0 \lrt U^M \st{\theta} \lrt M \st{\varphi^1} \lrt V^1 \lrt \cdots \st{\varphi^n} \lrt V^n \lrt 0
\end{equation}
 in $\M$, where $U \in \U$ and
\begin{equation}\label{Seq-1}
 0 \rightarrow \M(U,V^1) \rightarrow \M(U,V^2) \rightarrow \cdots \rightarrow \M(U,V^n) \rightarrow 0
\end{equation}
is exact, for all objects $U$ in $\CU$.

By applying the exact functor $\Pc$ on the sequence (\ref{seq}), we get the following exact sequence
\begin{equation}\label{S-1}
 0 \lrt \Pc(U^M) \st{\Pc(\theta)} \lrt \Pc(M) \st{\Pc(\varphi^1)} \lrt \Pc(V^1) \lrt \cdots \st{\Pc(\varphi^n)} \lrt \Pc(V^n) \lrt 0.
 \end{equation}
in $\Pc(\M)$.
To show that $\Pc(\U)$ is an $n$-torsion class of $\Pc(\M)$, it is enough to show that this is the canonical $n$-exact sequence of $\Pc(M)$ with respect to $\Pc(\U)$.

Since the sequence (\ref{S-1}) is exact, we may deduce from \cite[Lemma 3.5]{Iyama2011a}, that it is an  $n$-exact sequence in $\Pc(\M)$.

So it is enough to show that the sequence
\begin{equation}\label{eq:s2}
0 \lrt \Pc(\M)(\Pc(U), \Pc(V^1)) \lrt   \cdots \lrt \Pc(\M)(\Pc(U), \Pc(V^n)) \lrt 0
\end{equation}
is exact, for all objects $\Pc(U)$ of $\Pc(\CU)$.

Since the push-down functor $\Pc: \mmod \C~\lrt~\mmod\C/G$ is $G$-precovering, there exists the following commutative diagram
{\footnotesize{
\[
\begin{tikzcd}
0 \ar{r} & (\Pc(U), \Pc(V^1)) \ar{r}{\Pc(\theta)_*} & (\Pc(U), \Pc(V^2)) \ar{r} & \cdots \ar{r}{\Pc(\varphi^n)_* } & (\Pc(U), \Pc(V^n))  \ar{r} & 0 \\
0 \ar{r} & \oplus_{g \in G} \M({}^g U , V^1) \ar{u}{\Pc^{(1)}_{U, V^1}} \ar{r}{\oplus \theta_*}  & \oplus_{g \in G} \M({}^g U , V^2) \ar{u}{\Pc^{(1)}_{U, V^2}} \ar{r} & \cdots  \ar{r}{\oplus  \varphi^n_*} & \oplus_{g\in G} \M({}^g U, V^n) \ar{u}{\Pc^{(1)}_{U, V^n}} \ar{r}  & 0,
\end{tikzcd}
\]}}
where the vertical maps are $\k$-isomorphisms, see Definition \ref{k-iso}.

Since $\U$ is $G$-equivariant, ${}^g U$ belongs to $\U$ for all $g \in G$. Now, since the sequence (\ref{Seq-1}) is exact, for all $U \in \U$, the bottom row of the above diagram is exact. This implies the exactness of the top row, as desired.

This completes the proof of the theorem.
\end{proof}

In Theorem~\ref{thm:torsfromntors}, given an $n$-cluster-tilting subcategory $\M$ of an abelian category $\CA$, we characterise the minimal torsion class $T(\U)$ of $\CA$ containing the $n$-torsion class $\U \subset \M$.
In the following corollary we compare the minimal torsion class $T(\Pc(\U))$ of $\mmod\C/G$ containing $\Pc(\U)$ with the torsion class $\Pc(T(\U))$.
Recall that for a subcategory $\X$ of an abelian category $\CA$, the minimal torsion class of $\CA$ containing $\X$ is denoted by $T(\X)$ and is equal to $\Filt(\Fac(\X))$.

\begin{corollary}\label{cor:commutPT}
Let the situation be as in the Theorem \ref{thm:cov-main}.
Let $\CU$ be an $n$-torsion class of $\M$ and suppose that $T(\U)$ is $G$-equivariant.
Then $\Pc(T(\U))=T(\Pc(\U))$.
\end{corollary}

\begin{proof}
By definition,  $T(\U)=\Filt(\Fac(\U))$ is the minimal torsion class of $\mmod\C$ that contains $\U$. Let $M \in T(\U)$. So there exists a filtration
\[0=M_0 \subset M_1 \subset \dots \subset M_t=M\]
of $M$ such that $M_i/M_{i-1} \in \Fac(\U)$ for all $1 \leq i \leq t$. Since the push-down functor $\Pc$ is exact, we easily deduce that $\Pc(M) \in \Filt(\Fac(\Pc(\U)))=T(\Pc(\U))$. Therefore $\Pc(T(\U)) \subseteq T(\Pc(\U))$.
For the reverse inclusion, note that the inclusion $\U \subseteq T(\U)$ implies that $\Pc(\U) \subseteq \Pc(T(\U))$. Hence
\[T(\Pc(\U)) \subseteq T(\Pc(T(\U))).\]
But $ T(\Pc(T(\U)))=\Pc(T(\U))$, because by Theorem~\ref{thm:cov-main} we have that the functor $\Pc$ preserves torsion classes.
The proof is hence complete.
\end{proof}

Note that Theorem~\ref{thm:cov-main} implies that the functor $\Pc: \mmod\C \rt \mmod\C/G$ induces a map
$$n\Pc: \text{$G$-$n$-tors}(\M) \rightarrow n\text{-tors}(\Pc(\M))$$
from the set $\text{$G$-$n$-tors}(\M)$ of $G$-equivariant $n$-torsion classes of $\M \subset \mmod \C$ to the set  $n\text{-tors}(\Pc(\M))$ of $n$-torsion classes of $\Pc(\M) \subset \mmod\C/G$.
Likewise, $\Pc: \mmod\C \rt \mmod\C/G$ induces a map
$$\Pc: \text{$G$-tors}(\mmod\C) \rightarrow \text{tors}(\mmod\C/G)$$
from the set $\text{$G$-tors}(\M)$ of $G$-equivariant torsion classes of $\mmod \C$ to the set  $\text{tors}(\Pc(\M))$ of torsion classes of $\mmod\C/G$.
Using this notation, Corollary~\ref{cor:commutPT} can be restated as follows.

\[
\begin{tikzcd}
\text{$G$-$n$-tors}(\M)\ar{rr}{T(-)} \ar{dd}{n\Pc(-)}	& & \text{$G$-tors}(\mmod\C) \ar{dd}{\Pc(-)}\\
	& & \\
n\text{-tors}(\Pc(\M))\ar{rr}{T(-)}	& & \text{tors}(\mmod\C/G)\\
\end{tikzcd}
\]

\smallskip
Suppose that $\delta = \{\U_s : s\in [0,1]\}$ is a chain of $G$-equivariant $n$-torsion classes in $\M\subset \mmod\C$.
Then Theorem~\ref{thm:cov-main} implies that $\Pc(\delta):=\{\Pc(\U_s) : s\in [0,1]\}$ is a chain of $n$-torsion classes in $\Pc(\M) \subset \mmod\C/G$.

Now, Theorem~\ref{thm:nHNfilt} implies that for every non-zero object $M\in \M$, the chain of $G$-equivariant $n$-torsion classes $\delta$ induces an $n$-Harder-Narasimhan filtration, while $\Pc(\delta)$ induces an $n$-Harder-Narasimhan filtration of $\Pc(M)$.
In the following result we compare both filtrations.

\begin{proposition}\label{prop:pushdownHN}
Let  $\C$ be a locally bounded Krull-Schmidt $\k$-category with an admissible action of a group $G$ on $\C$ inducing an admissible action on $\mmod \C$.
Let $\M$ be a $G$-equivariant $n$-cluster-tilting subcategory of $\mmod\C$ such that $\Pc(\M)$ is functorially finite in $\mmod\C/G$.
Let $\delta=\{\U_s : s\in [0,1]\}$ be a chain of $G$-equivariant $n$-torsion classes in $\M$ and $M$ be a non-zero object of $\M$.
Then a filtration
$$0 =M_0 \subsetneq M_1 \subsetneq \dots \subsetneq M_{r-1} \subsetneq M_r = M$$
is the $n$-Harder-Narasimhan filtration of $M$ with respect to $\delta$ in $\M$ if and only if
$$0 =\Pc(M_0) \subsetneq \Pc(M_1) \subsetneq \dots \subsetneq \Pc(M_{r-1}) \subsetneq \Pc(M_r )= \Pc(M)$$
is the $n$-Harder-Narasimhan filtration of $\Pc(M)$ with respect to the chain of $n$-torsion classes $\Pc(\delta)$ in $\Pc(\M)$.
\end{proposition}

\begin{proof}
Clearly, the union and intersection of $G$-equivariant sets is itself $G$-equivariant.
This fact together with Proposition~\ref{prop:limits} imply that $\bigcup_{r> s} \U_r$ and $\bigcap_{t< s} \U_t$ are $G$-equivariant $n$-torsion classes for every $s \in [0,1]$.
Hence, it follows from Theorem~\ref{thm:cov-main} that $\Pc(\bigcup_{r> s} \U_r)$ and $\Pc(\bigcap_{t<s} \U_t)$ are $n$-torsion classes in $\Pc(\M)$.

We claim that $\Pc(\bigcup_{r> s} \U_r) = \bigcup_{r> s} \Pc(\U_r)$ and  $\Pc(\bigcap_{t<s} \U_t)= \bigcap_{t<s} \Pc(\U_t)$.
We only show the first of these equalities, the proof of the second being similar.

Clearly, $\U_r \subset \bigcup_{r> s} \U_r$ for all $s< r\leq 1$.
Then $\Pc(\U_r) \subset \Pc(\bigcup_{r> s} \U_r)$ for all $s< r\leq 1$.
Hence $\bigcup_{r>s}\Pc(\U_r) \subset \Pc(\bigcup_{r> s} \U_r)$.
For the reverse inclusion, let $X \in \Pc(\bigcup_{r> s} \U_r)$.
Then $X=\Pc(Y)$ for some $Y \in \bigcup_{r> s} \U_r$.
This implies the existence of a $r\in (s, 1]$ such that $Y \in \U_r$.
Thus $X=\Pc(Y)\in \Pc(\U_r) \subset \bigcup_{r>s}\U_r$ and our claim follows.

Let $M$ be a non-zero object of $\M$ and let
$$0 =M_0 \subsetneq M_1 \subsetneq \dots \subsetneq M_{r-1} \subsetneq M_r = M$$
be the $n$-Harder-Narasimhan filtration of $M$ with respect to $\delta$ in $\M$, it follows from Definition~\ref{def:n-slicing} that $M_{r-1}$ is the $n$-torsion object of $M$ with respect to the $n$-torsion class $\bigcup_{r>s_{r}}\U_r$, where $s_r= \sup\{t\in [0,1] : M\not\in \U_t \}$.
Then, Theorem~\ref{thm:cov-main} implies that $\Pc(M_{r-1})$ is the torsion object of $\Pc(M)$ with respect to the $n$-torsion class $\Pc(\bigcup_{r> s} \U_r) = \bigcup_{r> s} \Pc(\U_r)$.

Repeating this argument inductively, we obtain an $n$-Harder-Narasimhan filtration
$$0 =\Pc(M_0) \subsetneq \Pc(M_1) \subsetneq \dots \subsetneq \Pc(M_{r-1}) \subsetneq \Pc(M_r )= \Pc(M)$$
of $M$ with respect to the chain of torsion classes $\Pc(\delta)$.
Since the $n$-Harder-Narasimhan filtration is unique up to isomorphism by Theorem~\ref{thm:nHNfilt}, the proof is finished.
\end{proof}

To finish the paper we illustrate the results of this section.

\begin{example}
Let $B$ be the be the path algebra of the quiver
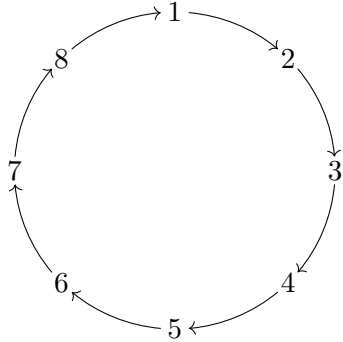
\begin{center}
\begin{tikzpicture}
\foreach \a in {1,2,...,8}{
\draw ( 360 - \a*360/8 + 135 : 60pt) node{\a};
}
\foreach \a in {1, 2, ..., 8}{
\draw[<-] (\a*360/8 + 50 : 60pt) arc (\a*360/8 + 50 : (\a+2)*360/8 - 5 :60pt); }
\end{tikzpicture}
\end{center}
modulo its radical squared.
The module category $\mmod B$ of $B$ has a $2$-cluster tilting subcategory
$$\M = \add \left\{ \Rep{1} \oplus \Rep{3} \oplus \Rep{5} \oplus \Rep{7} \oplus \Rep{8\\1} \oplus \Rep{1\\2} \oplus \Rep{2\\3} \oplus \Rep{3\\4} \oplus \Rep{4\\5} \oplus \Rep{5\\6} \oplus \Rep{6\\7} \oplus \Rep{7\\8} \right\}.$$
The Auslander-Reiten quiver of $B$ can be seen in Figure~\ref{fig:Ar-quiverB}.
In the figure the indecomposable objects that belong to $\M$ are indicated in red and the dotted arrows correspond to the Auslander-Reiten translation in $\mmod B$.

\begin{figure}[ht]
\centering
\begin{tikzpicture}[scale = 0.88]
\clip(-0.3, -0.3) rectangle (16.3, 1.8);

\foreach \a in {1, 3, 5, 7}{
\draw[color = red] (16 - 2*\a , 0) node{$\Rep{\a}$};}

\foreach \a in {2, 4, 6, 8}{
\draw (16 - 2*\a , 0) node{$\Rep{\a}$};}
\draw (16, 0) node{$\Rep{8}$};

\foreach \a in {1, 2, 3, ..., 7}{
\pgfmathsetmacro\result{\a + 1}
\draw[color = red] (15 - 2*\a , 1) node{$\Rep{\a \\ \pgfmathprintnumber{\result} }$};}
\draw[color = red] (15, 1) node{$\Rep{8\\1}$};

\foreach \a in {1, 2, ..., 8}{
\draw [->] (2*\a - 1.8 , 0.2) -- (2*\a - 1.2 , 0.8);}

\foreach \a in {1, 2, ..., 8}{
\draw [->] (2*\a - 0.8 , 0.8) -- (2*\a - 0.2 , 0.2);}

\foreach \a in {1, 2, ..., 8}{
\draw [<-, dashed] (2*\a - 1.8 , 0) -- (2*\a - 0.2 , 0);}

\end{tikzpicture}
\caption{The Auslander-Reiten quiver of $B$}
    \label{fig:Ar-quiverB}
\end{figure}
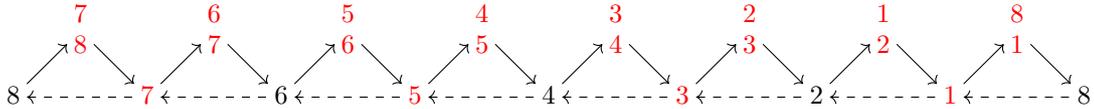

We note that the algebra $B$ can be seen as a {$\k$-linear} category $\C_B$ having exactly eight objects $\{e_1, \dots, e_8\}$ which are pairwise non-isomorphic such that $\Hom_{\C_B} (e_i, e_j)$   is non-empty if and only if $e_i B e_j$ is non-zero. Note that the identity morphism of the object $e_i$ corresponds to the element $e_i$  in $e_i B e_i$ and, moreover, $\Hom(e_i, e_j)$ is one-dimensional  corresponding to the one-dimensional vector space $e_j B e_i$ if there is an arrow from $j$ to $i$ or if $i = j$ and $\Hom(e_i, e_j)$ is the zero vector space otherwise.
It follows from the definitions that $\mmod \C_B$ is equivalent to $\mmod B$.

We also note that there is an admissible $\mathbb{Z}_2$ action $g$ over $\C_B$ which in the objects is defined as $g(e_i) = e_{i+4 \mmod 8}$.
In this case we have that $\C_B / \mathbb{Z}_2$ is a category having exactly four pairwise non-isomorphic objects and $\mmod \C_B / \mathbb{Z}_2$ is equivalent to the module category $\mmod C$, where $C$ is the path algebra of the following quiver
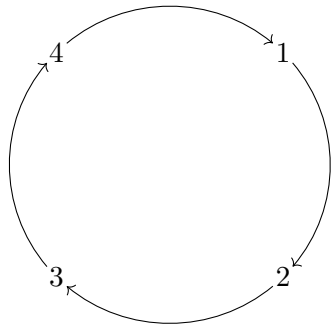
\begin{center}
\begin{tikzpicture}
\foreach \a in {1,2, 3,4}{
\draw ( 495 - \a*360/4  : 60pt) node{\a};
}
\foreach \a in {1,2,3,4}{
\draw[->] (\a*360/4 + 40 : 60pt) arc (\a*360/4 + 40 : (\a)*360/4 - 40 : 60pt); }
\end{tikzpicture}
\end{center}
module the radical squared.
In $\mmod C$ there is a $2$-cluster tilting subcategory $\M'$.
The Auslander-Reiten quiver of $\mmod C$ can be seen in Figure~\ref{fig:Ar-quiverC}.
In red we highlight the indecomposable objects of $\mmod C$ that are in $\M'$.

\begin{figure}[ht]
\centering
\begin{tikzpicture}[scale = 0.88]
\clip(-0.3, -0.3) rectangle (8.3, 1.8);

\foreach \a in {1, 3}{
\draw[color = red] (8 - 2*\a , 0) node{$\Rep{\a}$};}

\foreach \a in {2, 4}{
\draw (8 - 2*\a , 0) node{$\Rep{\a}$};}
\draw (8, 0) node{$\Rep{4}$};

\foreach \a in {1, 2, 3}{
\pgfmathsetmacro\result{\a + 1}
\draw[color = red] (7 - 2*\a , 1) node{$\Rep{\a \\ \pgfmathprintnumber{\result} }$};}
\draw[color = red] (7, 1) node{$\Rep{4\\1}$};

\foreach \a in {1, 2, ..., 4}{
\draw [->] (2*\a - 1.8 , 0.2) -- (2*\a - 1.2 , 0.8);}

\foreach \a in {1, 2, ..., 4}{
\draw [->] (2*\a - 0.8 , 0.8) -- (2*\a - 0.2 , 0.2);}

\foreach \a in {1, 2, ..., 4}{
\draw [<-, dashed] (2*\a - 1.8 , 0) -- (2*\a - 0.2 , 0);}

\end{tikzpicture}
\caption{The Auslander-Reiten quiver of $C$}
    \label{fig:Ar-quiverC}
\end{figure}

As mentioned in the introduction, there is a natural push-down functor $\Pc : \mmod B \rt \mmod C$. Moreover, it follows from the results of \cite{DarpoIyama} that $\Pc(\M)=\M'$. Now, we know from Theorem~\ref{thm:cov-main} that for any $\mathbb{Z}_2$-equivariant $2$-torsion class $\U$ of $\M$, $\Pc(\U)$ is a $2$-torsion class in $\M'$. In the following table, we give a complete list of all $\mathbb{Z}_2$-equivariant $2$-torsion classes of $\M$ and their respective image under $\Pc : \mmod C \rt \mmod B$, where we denote the set of $\mathbb{Z}_2$-equivariant $2$-torsion classes of $\M$ by $\mathbb{Z}_2$-$2$-tors($\M$) and the set of $2$-torsion classes of $\M'$ by  $2$-tors($\M'$).

\begin{center}
\begin{tabular}{|| c  | c ||}
\hline
$\U\in \mathbb{Z}_2$-$2$-tors($\M$) & $\Pc(\U) \in$ $2$-tors($\M'$)  \\
\hline
$\M$	& $\M'$  \\
\hline
$\add \left\{ \Rep{6\\7} \oplus \Rep{5\\6} \oplus \Rep{5} \oplus \Rep{2\\3} \oplus \Rep{1\\2} \oplus \Rep{1} \right\}$	&
$\add \left\{ \Rep{2\\3} \oplus \Rep{1\\2} \oplus \Rep{1} \right\}$
\\
\hline
$\add \left\{ \Rep{8\\1} \oplus \Rep{7\\8} \oplus \Rep{7} \oplus \Rep{4\\5} \oplus \Rep{3\\4} \oplus \Rep{3} \right\}$	&
$\add \left\{ \Rep{4\\1} \oplus \Rep{3\\4} \oplus \Rep{3} \right\}$
\\
\hline
$\add \left\{ \Rep{5\\6} \oplus \Rep{5} \oplus \Rep{1\\2} \oplus \Rep{1} \right\}$	&
$\add \left\{ \Rep{1\\2} \oplus \Rep{1} \right\}$
\\
\hline
$\add \left\{ \Rep{7\\8} \oplus \Rep{7} \oplus \Rep{3\\4} \oplus \Rep{3} \right\}$	&
$\add \left\{ \Rep{3\\4} \oplus \Rep{3} \right\}$
\\
\hline
$\add \left\{  \Rep{5} \oplus \Rep{1} \right\}$	&
$\add \left\{ \Rep{1} \right\}$
\\
\hline
$\add \left\{ \Rep{7} \oplus \Rep{3} \right\}$	&
$\add \left\{ \Rep{3} \right\}$
\\
\hline
$\add \left\{ 0 \right\}$	&
$\add \left\{ 0 \right\}$ \\
\hline
\end{tabular}
\end{center}

Now consider the chain $\delta$ of $\mathbb{Z}_2$-equivariant $2$-torsion classes of $\M$ defined as follows.
$$\delta=
\begin{cases}
\U_t = \M  & \text{ if $t \in [0, 1/3)$}\\
\U_t = \add \left\{ \Rep{6\\7} \oplus \Rep{5\\6} \oplus \Rep{5} \oplus \Rep{2\\3} \oplus \Rep{1\\2} \oplus \Rep{1} \right\}  & \text{ if $t \in [1/3, 2/3)$}\\
\U_t = \add\{0\}  & \text{ if $t \in [2/3, 1]$}
\end{cases}
$$
An easy calculation shows that $\Pc(\delta)$ is the following chain of $2$-torsion classes in $\M'$.
$$\Pc(\delta)=
\begin{cases}
\Pc(\U_t) = \M  & \text{ if $t \in [0, 1/3)$}\\
\Pc(\U_t) = \add \left\{ \Rep{2\\3} \oplus \Rep{1\\2} \oplus \Rep{1} \right\}  & \text{ if $t \in [1/3, 2/3)$}\\
\Pc(\U_t) = \add\{0\}  & \text{ if $t \in [2/3, 1]$}
\end{cases}
$$

If we take the object $\rep{8\\1}\oplus \rep{4\\5}$, one can see that its $2$-Harder-Narasimhan filtration with respect to $\delta$ is $0 \subset \rep{1}\oplus \rep{5} \subset \rep{8\\1}\oplus \rep{4\\5}$.
Moreover, we have that $\Pc\left(\rep{8\\1}\oplus \rep{4\\5}\right) = \rep{4\\1} \in \M'$.
Calculating the $2$-Harder-Narasimhan filtration of $\rep{4\\1}$ with respect to $\Pc(\delta)$ we obtain $0 \subset \rep{1} \subset \rep{4\\1}$, where we can see that $\rep{1} = \Pc(\rep{1}\oplus \rep{5})$ as shown in Proposition~\ref{prop:pushdownHN}.
\end{example}

\def\cprime{$'$}

\makeaddress
\end{document}